\documentclass[10pt]{amsart}
\usepackage{amsmath}
\usepackage{amsfonts}
\usepackage{amssymb}
\usepackage{graphicx}
\usepackage{amsthm,graphicx,color,yfonts}
\usepackage{pdfsync}
\usepackage{epstopdf}

\usepackage[colorlinks=true]{hyperref}
\hypersetup{linkcolor=red,citecolor=blue,filecolor=dullmagenta,urlcolor=darkblue} 

\numberwithin{equation}{section}

\theoremstyle{plain}
\newtheorem{theorem}{Theorem}[section]

\newtheorem{lemma}[theorem]{Lemma}
\newtheorem*{de-lemma}{Lemma}
\newtheorem{proposition}[theorem]{Proposition}

\theoremstyle{remark}

\newtheorem*{remark}{Remark}

\theoremstyle{definition}

\newcommand{\Z}{\mathbb{Z}}
\newcommand{\dd}{\mathrm{d}}
\newcommand{\R}{\mathbb{R}}

\newcommand{\ve}{\epsilon}

\providecommand{\MR}{\relax\ifhmode\unskip\space\fi MR }

\providecommand{\href}[2]{#2}

\begin{document}

\title[Phase transition model]{Energy minimizers in a periodic phase transition model of light-matter interaction in nematic liquid crystals}

\author{Panayotis Smyrnelis}
\address{Department of Mathematics, University of Athens, 15784 Athens, Greece},

\email{smpanos@math.uoa.gr}

\author{Marcel G. Clerc}
\address{Departamento de F\'{i}sica and Millennium Institute for Research in Optics,
FCFM, Universidad de Chile, Casilla 487-3, Santiago, Chile, FCFM}
\email{marcel@dfi.uchile.cl}

\author{Manuel Diaz-Zuniga}
\address{Departamento de F\'{i}sica and Millennium Institute for Research in Optics,
FCFM, Universidad de Chile, Casilla 487-3, Santiago, Chile, FCFM}
\email{manudiaz@ug.uchile.cl}

\author{Micha{\l } Kowalczyk}
\address{Departamento de Ingenier\'{\i}a Matem\'atica and Centro
de Modelamiento Matem\'atico (UMI 2807 CNRS), Universidad de Chile, Casilla
170 Correo 3, Santiago, Chile.}
\email {kowalczy@dim.uchile.cl}



\begin{abstract} In this paper we complete the study of global minimizers of a forced, non autonomous, one dimensional, phase transition 
model, initiated in \cite{kink}. Motivated by the recent findings in \cite{opt}, revealing new configurations of topological structures in
light, we consider a forcing term having two periods. We show that depending on the strength of the forcing, at most two thresholds that determine the structure of the minimizers (kinks) are attained. These kinks are now a combination of three types observed experimentally, and they may have at most three zeros. The existence of these complex types of phase transition follows from a periodic one dimensional model 
of matter-light interaction in nematic liquid crystal based on a thin sample limit  of the Oseen-Frank energy. We show that the qualitative behaviour of  global minimizers  is consistent with the original model.

\textbf{Keywords}:  minimizers, kink, periodic, phase transition, liquid crystals, light-matter interaction

\textbf{MSC[2020]}:  35J20, 35J61, 35Q56, 35Q60

\end{abstract}
\maketitle

\section{Introduction}

\subsection{Physical motivation}
Optical vortices have garnered attention because of their unique topological characteristics and
applications ranging from communications \cite{m2,m3}, astronomical \cite{m4} and microscopic image processing \cite{m5,m6}, quantum computing \cite{m6,m7}, orbital angular momentum lasers \cite{m8}, to manipulating micrometric particles \cite{m9,m10,m11}. Various methods have been proposed for vortex
generation. For instance, in a suitable experimental set up \cite{Barboza2012, Barboza2015A, LCLV-Vortex2015,PhysRevLett.111.093902, clerc2} involving a liquid crystal sample, a laser and a
photoconducting cell one can observe light defects such as kinks, domain walls and vortices. In most of these devices a Gaussian light beam is used to illuminate the liquid crystal light valve, and induce an umbilical defect in
the liquid crystal film within the illuminated central region
\cite{LCLV-Vortex2015}. The main feature of this system is that as the light intensity varies, the central vortex (which is called \emph{standard vortex}) 
becomes unstable and shifts to the non-illuminated
area, resulting in a \emph{shadow vortex} \cite{clerc2}. By deriving an amplitude equation from first principles, we could figure out the origin and relationships of these vortices, and confirm theoretically these experimental observations \cite{symb}.

The present paper is motivated by some recent findings \cite{opt} revealing new configurations of complex topological structures in
light, generated by illuminations with other properties. Namely, by using a donut-type light beam, the system induces a triad vortex, called \emph{vortex rosette}, with configurations in its equilibrium state of various types (standard, shadow,
and giant\footnote{In subsections \ref{sec:mathem} and \ref{sec:mainsec}, the features of these three types of vortices will be discussed.}). This rosette consists of a giant vortex at the center, surrounded
by two standard or shadow vortices in the illuminated
ring. Again, the shift between the two standard or shadow vortices is due to the variation of the light intensity. 

The scope of the paper is to explain mathematically these new results.

\subsection{The mathematical model}\label{sec:mathem}

To describe the energy of  the illuminated liquid crystal light valve, we consider  the Oseen-Frank model in the vicinity of the
Fr\'{e}edericksz transition. It is explained in \cite{kink} how this problem can be reduced to the study of the following  Ginzburg-Landau type functional: 
\begin{equation}
\label{funct00}
E(u)=\int_{\R^2}\frac{\epsilon}{2}|\nabla u|^2-\frac{1}{2\epsilon}\mu(x)|u|^2+\frac{1}{4\epsilon}|u|^4-\alpha f(x)\cdot u,
\end{equation}
where $u=(u_1,u_2)\in  H^1(\R^2,\R^2)$ and $\epsilon>0$, $\alpha\geq 0$ are real parameters. In the physical context, the functions $\mu$ and $f$ are specific:
\begin{subequations}
\begin{equation}\label{beam1}
\mu(x)=e^{\,-|x|^2}-\chi, \qquad \mbox{with some}\ \chi\in (0,1), \qquad f(x)=-\frac{1}{2}\nabla \mu(x),
\end{equation} 
for a light beam with a Gaussian intensity profile, as described in \cite{clerc2}, or
\begin{equation}\label{beam2}
\mu(x)=|x|^4e^{\,-|x|^2}-\chi, \qquad \mbox{with some}\ \chi\in (0,4 e^{-2}), \qquad f(x)=-\frac{1}{2}\nabla \mu(x), 
\end{equation}
\end{subequations}
for a light beam with a donut-shaped intensity profile, as described in \cite{opt}, which is the case under study in the present paper.

Physically, the order parameter  $u$  represents the intensity of  light induced by the  interaction between the  laser beam of profile given by $\mu$, and the nematic liquid crystal sample with the photoconducting cell mounted on top of it. This cell generates  electric field whose small, vertical component is described  above by $f$. The parameter $\alpha$ is non dimensional and characterizes the intensity of the laser beam.

In our previous work \cite{symb}, we already established that in the context of the Gaussian light beam (cf. \eqref{beam1}), the structures observed experimentally have their analogues in the global  minimizers of \eqref{funct00}. In particular, we could prove that
\begin{itemize}
\item when $\alpha=0$ the global minimizer is a vortex free state, 
\item when $\alpha\sim \epsilon|\ln\epsilon|$, global minimizers have (at least) one vortex which is located at the border of the illuminated region, where $\mu\simeq 0$. It is called a \emph{shadow} vortex. Additionally, its amplitude is very small, of order $\mathcal O(\epsilon^{1/3})$, while its width is of order $\mathcal O(\epsilon^{2/3})$.
\item when $\alpha\sim \epsilon|\ln\epsilon|^2$ global minimizers have one vortex that looks like the standard Ginzburg-Landau vortex, and is located near $0$, i.e. at the center of the illuminated region $\{\mu>0\}$. It is called a \emph{standard} vortex. Additionally, its amplitude is $\sqrt{\mu(0)}$, while its width is of order $\mathcal O(\epsilon)$.
\end{itemize}

Actually, before studying the minimizers of the two dimensional energy functional \eqref{funct00}, we focused on the equivalent one dimensional model, described by the functional  
\begin{equation}
\label{funct 0}
E(u)=\int_{\R}\frac{\epsilon}{2}|u'|^2-\frac{1}{2\epsilon}\mu(x)u^2+\frac{1}{4\epsilon}|u|^4-\alpha f(x)u,
\end{equation}
where $u\in  H^1(\R)$, while $\epsilon>0$, and $\alpha\geq 0$ are real parameters. In the context of a Gaussian light beam (i.e. for a profile of $\mu$ similar to $\mu(x)=e^{\,-x^2}-\chi$, with some $\chi\in (0,1)$, and  $f(x)=-\frac{1}{2}\mu'(x)$), we could obtain a very complete description of global minimizers. 
In contrast with the two dimensional case, we could establish that when $\alpha>0$, the minimizers $v_{\epsilon,\alpha}$ have exactly one transition $\bar x$ (i.e. one zero).  Moreover, the aforementioned regimes are now observed as $\epsilon \to 0$, while $\alpha$ remains fixed:
\begin{itemize}
\item If $\alpha\in (0,\sqrt{2})$, we have a \emph{shadow kink} whose zero $\bar x$ is close to one of the two points $\pm\xi$ where $\mu$ vanishes. Its rescaled profile $\frac{v_\epsilon(\xi+s\epsilon^{2/3})}{\epsilon^{1/3}}$ in a neighborhood of $\xi$, converges as $\epsilon\to 0$, to a solution of the Painlev\'e equation \eqref{oderes4-0x} (cf. also \cite{sourdis}).
\item as soon as $\alpha$ passes the critical value $\sqrt{2}$ from below,  $\bar x$ moves to $0$, and we have a \emph{standard kink} whose rescaled profile $v_\epsilon(\bar x+\epsilon s)\simeq\sqrt{\mu(0)}\tanh(s\sqrt{\mu(0)/2})$ in a neighborhood of $\bar x$, converges as $\epsilon\to 0$, to the heteroclinic orbit of the O.D.E. $V''=V^3-\mu(0)V$. 
\end{itemize}
It is remarkable that despite the symmetry of the problem, the shadow kink does not preserve it.
\subsection{Main results}\label{sec:mainsec}

Our scope is to complete the study initiated in \cite{kink}, and investigate the behaviour of global minimizers of \eqref{funct 0}, in the case of the donut-shaped light beam, i.e. for  
\begin{equation}\label{beam3}
\mu(x)=x^4e^{\,-x^2}-\chi, \qquad \mbox{with some}\ \chi\in (0,4 e^{-2}), \qquad f(x)=-\frac{1}{2}\mu'(x). 
\end{equation}
Actually, all our results hold under more general hypothesis on $\mu$ and $f$ which we will state now. We consider a smooth, even and bounded function $\mu$ having $2$ bumps, and such that the set $\{\mu \geq 0\}$ is bounded. More precisely, we assume that:
\begin{equation}
\label{hyp3}
\begin{cases}
 \text{$\mu \in C^1(\R) \cap L^\infty(\R)$ is even,  $\mu'<0$ on $(\zeta,\infty)$, for some $\zeta>0$, $\mu(\xi)=0$ for a unique $\xi>\zeta$,}\\
\text{$\mu'>0$ on $(0,\zeta)$,}\\
\text{and  $f:=-\frac{\mu'}{2}\in L^\infty(\R)$, $\liminf_{x\to 0}\frac{|f(x)|}{|x|^\beta}>0$, for some $\beta>0$.}
\end{cases}
\end{equation}
As a consequence of these assumptions, $\mu$ attains its maximum at the points $\pm\zeta$, while at the point $0$, it has a local minimum. Motivated by the physical model 
of matter-light interaction in nematic liquid crystals, we take $f=-\frac{\mu'}{2}$, which implies that $f\in L^1(\R)$. On the other hand, the assumption $\liminf_{x\to 0}\frac{|f(x)|}{|x|^\beta}>0$ will only be used in the proof of Lemma \ref{p3bbis} to bound the zero of the minimizer which is close to $0$. In addition, we assume one of the following:
\begin{itemize}
 \item[(A)] $\mu$ remains positive in the interval $[0,\zeta]$ (i.e. within a period).
 \item[(B)] $\mu$ changes sign in the interval $[0,\zeta]$ (i.e. between the two periods): we denote by $\xi'$ its unique zero in $(0,\zeta)$.
\end{itemize}
\begin{figure}[h]
\includegraphics[width=.6\textwidth]{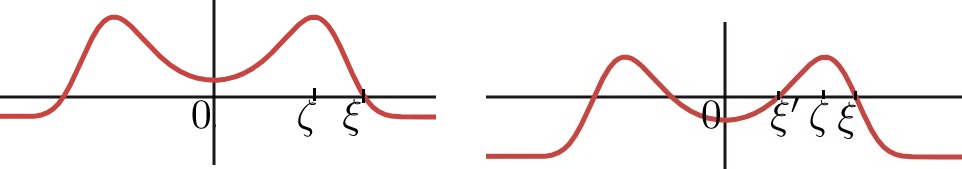}
\caption{The function $\mu$ in cases (A) (left), and (B) (right).}
\label{fig1}
\end{figure}
In this setting, the illuminated region $\{\mu(x)>0\}$ is the union of one (case (A)) or two segments (case (B))\footnote{For the donut-shaped light beam in \eqref{beam3}, the illuminated region is a ring.}.
For such functions $\mu$ with $N=2$ bumps, describing the number of transitions of the minimizers and their location, becomes a much more involved problem than the case $N=1$ examined in \cite{kink}. We show (cf. Proposition \ref{p3}) that since the function $f=-\frac{\mu'}{2}$ changes sign $3$ times, the minimizers have at most $3$ zeros. These minimizers solve the Euler-Lagrange equation of $E$:
\begin{equation}\label{digital}
\ve^2 u''+\mu(x) u-u^3+\ve \alpha f(x)=0,\qquad x\in \R,
\end{equation} 
and form patterns combining different kinds of elementary kinks. In particular, in case (B), in addition to the shadow or standard kink encountered in \cite{kink}, there is also a transition at a point $\bar x$ close to $0$, which is called \emph{giant kink}, since it spreads all over the segment $(-\xi',\xi')$, where $\mu<0$. Its profile near $\bar x$ is given by $v(\bar x+h)\simeq -\frac{\epsilon\alpha f(h)}{\mu(h)}$, thus its amplitude is of order $\mathcal O(\epsilon)$.

Cases (A) and (B) induce different regimes ruling the number of zeros of the minimizers as well as  their convergence as $\epsilon\to 0$, according to the values of $\alpha>0$.
In case (B), we show in Proposition \ref{p3ter} that the global minimizers have exactly $3$ zeros. 
When $0<\alpha<\sqrt{2}$, the graph of a global minimizer remains close to the graph of $\pm\sqrt{\mu^+}$: two zeros approach the set $\{\pm\xi,\pm\xi'\}$ where $\mu$ vanishes (shadow kinks), while the third one goes to $0$ (giant kink). 
On the other hand, when $\alpha>\sqrt{2}$, two zeros go to $\pm\zeta$ (standard kinks), the maxima of $\mu$, while the third one still goes to $0$ (giant kink). For both regimes, the graphs of global minimizers look like periodic, reproducing a similar pattern on  the intervals $[-\xi,0]$ and $[0,\xi]$.

\begin{figure}[h]
\includegraphics[width=.6\textwidth]{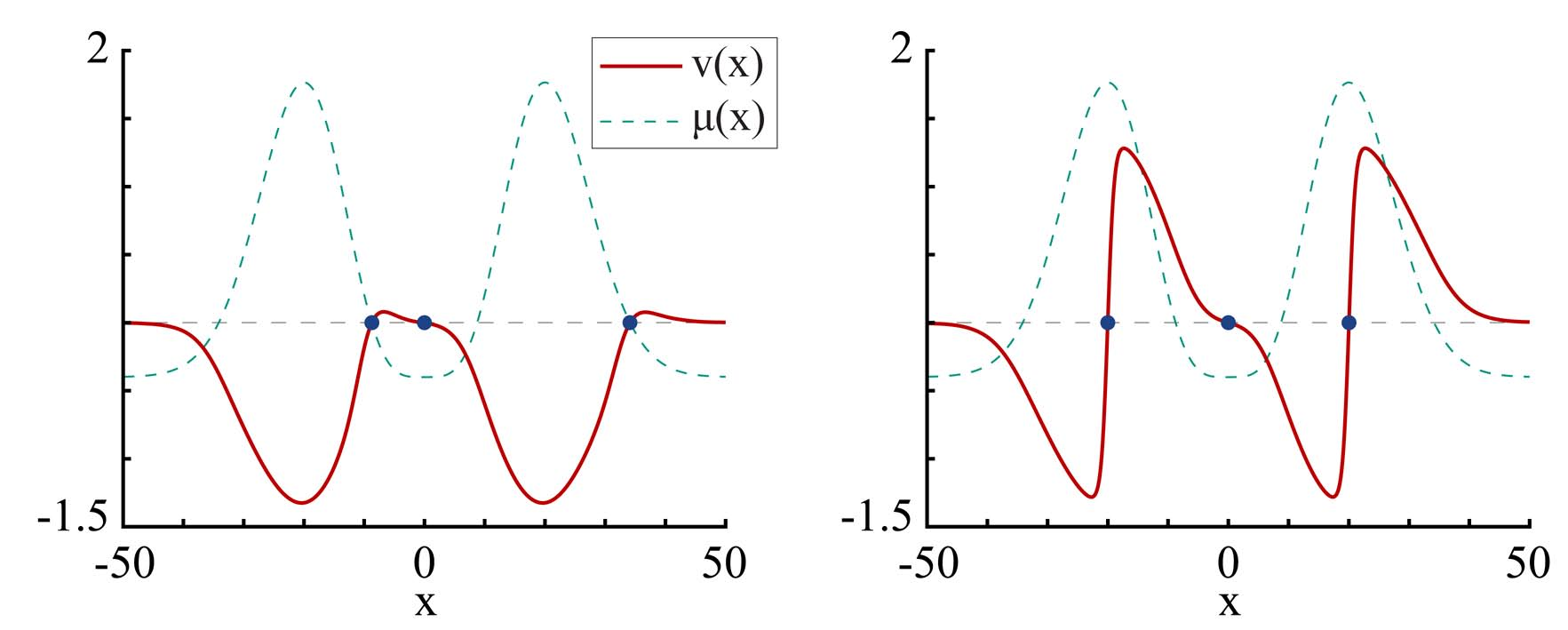}
\caption{The regimes $0<\alpha<\sqrt{2}$ (left), and $\alpha>\sqrt{2}$ (right), ruling the behaviour of global minimizers $v$ in case (B).}
\label{fig1}
\end{figure}
In case (A) however, in addition to the first threshold $\alpha^*=\sqrt{2}$, there is a second one $\alpha^{**}>\alpha^*$, that also determines the behaviour and the number of zeros of global minimizers.
We will see that as $\epsilon\to 0$, global minimizers have exactly one zero going respectively to $\pm\xi$ (shadow kink) when $0<\alpha<\alpha^*$, and to $\pm\zeta$ (standard kink) when $\alpha^*<\alpha<\alpha^{**}$. 
Finally, when $\alpha>\alpha^{**}$, global minimizers have exactly three zeros (standard kinks) going to $-\zeta$, $0$ and $\zeta$. Case (A) is thus more complex than case (B): after the first threshold, global minimizers start moving away from the graph of $\pm\sqrt{\mu^+}$, but they reproduce a periodic pattern only after the second threshold.
\begin{figure}[h]
\includegraphics[width=.9\textwidth]{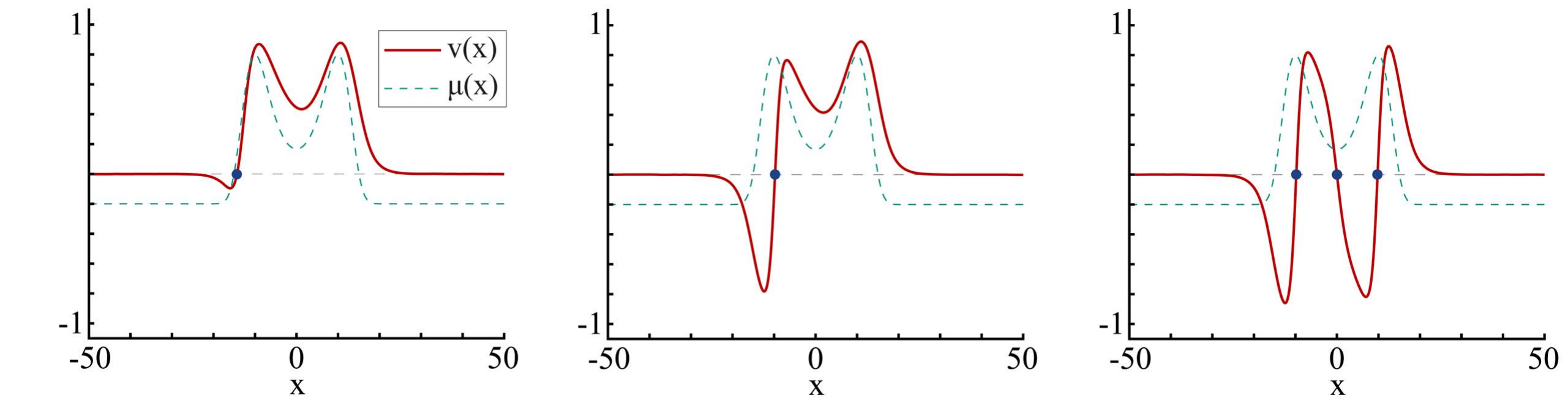}
\caption{The regimes $0<\alpha<\sqrt{2}$ (left), $\alpha^*<\alpha<\alpha^{**}$ (center), and $\alpha>\alpha^{**}$ (right), ruling the behaviour of global minimizers $v$ in case (A).}
\label{fig2bb} 
\end{figure}

To prove these results (cf. Theorems \ref{p10sh} and \ref{p10shter}), we first define the renormalized energy as 
\begin{align}\label{renorm}
\mathcal{E}(u)&:=E(u)+\int_{\mu>0}\frac{\mu^2}{4\epsilon}=\int_{\R}\frac{\epsilon}{2}|u'|^2+\int_{\mu>0}
\frac{(u^2-\mu)^2}{4\epsilon}+\int_{\mu<0}\frac{u^2(u^2-2\mu)}{4\epsilon}-\int_\R \alpha fu.
\end{align}  
Indeed, the minimum of the energy defined in \eqref{funct 0} tends to $-\infty$ as $\epsilon \to 0$ , thus it is more convenient to work with functional \eqref{renorm}, which is obtained by adding to \eqref{funct 0} a suitable term so that the result is bounded from below and above by an $\ve$ independent constant. Next, we compare the upper bounds obtained for $\mathcal{E}$ in Section \ref{sec:renormalized}, with the lower bounds corresponding to all possible configurations of zeros. This allows us to deduce the location of zeros, except for threshold values of $\alpha$, since in that case the aforementioned regimes are energetically equivalent. 

When $\mathcal{E}$ is minimized, there is a competition between the terms appearing in \eqref{renorm}, since each one has a different effect. Intuitively, one can see that the term $\frac{\epsilon}{2}|u'|^2$ penalizes high variations of minimizers, while the terms $\frac{(u^2-\mu)^2}{4\epsilon}$ on $\{\mu>0\}$
and $\frac{u^2(u^2-2\mu)}{4\epsilon}$ on $\{\mu<0\}$ force minimizers to remain close to $\pm\sqrt{\mu^+}$. Finally, the effect of the last term  $-\alpha fu$ increases with $\alpha$. It is minimized when $f$ and $u$ have the same sign, i.e. when the transitions occur at $0$ and $\pm\zeta$. This it what happens for $\alpha$ large, and actually, when $\alpha$ is large enough the global minimizer is odd and unique (cf. Theorem \ref{p4d}). Thus, the symmetry (initially broken for small $\alpha>0$) is restored.

For functions $\mu$ with $N\geq 3$ bumps, the same regimes mentioned above for $N=2$, are observed in both cases 
(A) and (B). Additionally, the value of the second threshold $\alpha^{**}$ is independent of $N$. That is why, and for the sake of simplicity, it is sufficient to focus on the case $N=2$. However, we will also examine in Section \ref{sec:per}, the case of periodic functions $\mu$, i.e. the limiting case $N\to\infty$.
Since the corresponding results are similar with the case $N=2$, we will only point out the new elements and give the proofs that are not self-evident.

Although the one dimensional model is simpler than the two dimensional one, it shows an excellent agreement with the experiments performed in \cite{opt}. Indeed, the triad vortex observed in \cite{opt} has its counterpart in the three zeros of the global minimizers of \eqref{funct 0}, that behave in a similar way. Moreover, our results obtained under assumption (A) above, may be exploited to carry out new experiments generated by illuminations with different properties (e.g. by taking $\mu(x)=e^{-(x^2-1)^2}-\chi$, with $\chi\in(0, e^{-1})$).

\section{Basic results for solutions and minimizers in the case of a two bump function $\mu$}
\label{sec:main-results}
In this section, we state some basic results for equation \eqref{digital}, and firstly the existence of minimizers of functional \eqref{funct 0}, which follows from standard variational arguments (cf. the proof of \cite[Lemma 2.1]{kink}). 

\begin{theorem}\label{th1}
Assuming \eqref{hyp3}, there exists $v \in H^1(\R)$ such that $E(v)=\min_{H^1(\R)} E$. In addition, $v\in C^2(\R)$ is a classical solution of O.D.E. \eqref{digital}.
There also exists $u \in H^1_{\mathrm{ odd}}(\R):=\{u\in H^1(\R): \ u \text{ is odd}\}$ such that $E(u)=\min_{H^1_{\mathrm{ odd}}(\R)} E$ and $u$ solves O.D.E. \eqref{digital}. 
\end{theorem}

In the sequel, we will always denote by $v$ or $v_{\epsilon,\alpha}$ the global minimizers, and by $u$ or $u_{\epsilon,\alpha}$, the odd ones. We will also denote by $E(u,(a,b))$ the integral of functional \eqref{funct 0} over an interval $(a,b)$.
We state next some general properties of solutions under assumptions \eqref{hyp3}. 

\begin{proposition}\label{p0}(Properties of solutions)
\begin{itemize}
\item[(i)] If $u$ is a solution of \eqref{digital} belonging to $H^1(\R)$, and $\phi\in H^1(\R)$, then
\begin{subequations}
\begin{equation}\label{subaa}
E(u)=\int_{\R}-\frac{1}{4\epsilon}|u|^4-\frac{\alpha}{2} f(x)u,
\end{equation}
\begin{equation}\label{subbb}
E(u+\phi)-E(u)=\int_{\R}\Big(\frac{\epsilon}{2}|\phi'|^2+\frac{(u^2-\mu)}{2\epsilon}\phi^2+\frac{(2u+\phi)^2\phi^2}{4\epsilon}\Big).
\end{equation}
\end{subequations}
\item[(ii)] For $\epsilon$ and $\alpha$ belonging to a bounded interval, let $u_{\epsilon,\alpha}$ be a solution of \eqref{digital} converging to $0$ at $\pm\infty$. Then, the solutions $u_{\epsilon,\alpha}$ are uniformly bounded by a constant $M$. Moreover, given an interval $I\subset\{x\in\R:\mu(x)<0\}$ such that $c^2:=\inf_I(-\mu)>0$, it holds
\begin{equation}\label{exp11a}
\frac{|u_{\epsilon,\alpha(x)}|}{\epsilon}\leq \frac{M}{\epsilon}e^{-\frac{c'}{\epsilon}d(x,\partial I)}+\frac{\alpha\|f\|_{L^\infty(\R)}}{c^2}\leq  \frac{Me^{-1}}{c'd(x,\partial I)}+\frac{\alpha}{c^2}\|f\|_{L^\infty(\R)},
\end{equation}
for a constant $c'\in(0,c)$ when $I$ is bounded, and $c'=c$ otherwise. 
 \end{itemize}
\end{proposition}
\begin{proof}

(i)  For solutions $u\in H^1(\R)$ of \eqref{digital}, it holds:
\begin{equation}\label{euler}
\int_\R -\epsilon^2u'\phi'+\mu u\phi-u^3\phi+\epsilon\alpha f\phi=0, \ \forall \phi\in H^1(\R).
\end{equation}
Choosing $\phi=u$ in \eqref{euler}, we obtain $\int_{\R}\frac{\epsilon}{2}|u'|^2-\frac{1}{2\epsilon}\mu(x)u^2=\int_{\R}-\frac{1}{2\epsilon}|u|^4+\frac{\alpha}{2} f(x)u$, and then \eqref{subaa} is straightforward. On the other hand, we have 
\begin{equation*}
E(u+\phi)-E(u)=\int_{\R}\Big(\frac{\epsilon}{2}|\phi'|^2+\frac{(3u^2-\mu)}{2\epsilon}\phi^2+\frac{u}{\epsilon}\phi^3+\frac{1}{4\epsilon}\phi^4\Big)+\int_\R  \Big(\epsilon u'\phi'-\frac{\mu u\phi}{\epsilon}+\frac{u^3\phi}{\epsilon}-\alpha f\phi\Big),
\end{equation*}
where the second integral vanishes in view of \eqref{euler}. This proves \eqref{subbb}.

(ii) For the uniform bound  $\|u_{\epsilon,\alpha}\|_{L^\infty}(\R)\leq M$, we refer to the proof of \cite[Lemma 2.2]{kink}.  Next, we consider solutions $u=u_{\epsilon,\alpha}$ satisfying this bound. In view of the Kato inequality we have
\begin{equation}\label{exp22a}
 |u|''\geq \big(\frac{c}{\epsilon}\big)^2|u|-\frac{\alpha}{\epsilon}\|f\|_{L^\infty(\R)}\Longrightarrow  \big(|u|-\frac{\epsilon\alpha}{c^2}\|f\|_{L^\infty(\R)}\big)''\geq  \big(\frac{c}{\epsilon}\big)^2 \big(|u|-\frac{\epsilon\alpha}{c^2}\|f\|_{L^\infty(\R)}\big),
\end{equation}
on $I$ in the distributional sense. Then, by a standard comparison argument (cf. for instance \cite[Lemma A.1]{book}), we obtain
\begin{equation}\label{exp22a}
 |u|-\frac{\epsilon\alpha}{c^2}\|f\|_{L^\infty(\R)}\leq M e^{-\frac{c'}{\epsilon}d(x,\partial I)}\Longrightarrow  \frac{|u|}{\epsilon}     \leq \frac{M}{\epsilon} e^{-\frac{c'}{\epsilon}d(x,\partial I)}+\frac{\alpha}{c^2}\|f\|_{L^\infty(\R)}\leq  \frac{Me^{-1}}{c'd(x,\partial I)}+\frac{\alpha}{c^2}\|f\|_{L^\infty(\R)},
\end{equation}
where the constant $c'\in(0,c]$ is as in the statement of the lemma.
\end{proof}

Under the symmetries in \eqref{hyp3}, the energy \eqref{funct 0} and equation 
\eqref{digital} are invariant under the odd symmetry $\hat v(x)=-v(-x)$. When $\alpha>0$ and $0<\epsilon\ll 1$, the global minimizer $v$ is very likely not unique up to this symmetry\footnote{We shall see in Theorems \ref{p4c} and \ref{p4d}, that global minimizers are unique for large $\epsilon$, or for large $\alpha$, with $\epsilon$ fixed.}. However, we prove in Proposition \ref{p4} that two distinct global minimizers do not intersect, and an odd minimizer $u$ either lies between $v$ and $\hat v$ or coincides with them (cf. Figure \ref{fig12b}). When $\alpha=0$, the uniqueness of the global and odd minimizers is ensured, up to change of sign:

\begin{proposition}\label{p1}(Properties of minimizers when $\alpha=0$)

When $\alpha=0$, and $0<\epsilon\ll 1$, let $v$ and $u$ be the minimizers constructed in Theorem \ref{th1}. Then,  
\begin{itemize}
\item[(i)]  $v$ is even, positive and unique, up to change of $v$ by $-v$.
\item[(ii)] $v'(x)<0$, $\forall x\geq \xi$, and $\max_{\R}v\leq\sqrt{ \mu(\zeta)}$,
\item[(iii)] $v(x)\leq v(x_0)e^{-\frac{\sqrt{-\mu(x_0)}}{\epsilon}(x-x_0)}$, $\forall x \geq x_0 >\xi$.
\end{itemize}
Similarly, $u$ is positive on $(0,\infty)$ and unique, up to change of $u$ by $-u$, and it also satisfies (ii)--(iii) above.
\end{proposition}
\begin{proof}
(i) Positivity and evenness are established as in the proof of \cite[Theorem 1,1]{kink} (cf. \cite[Step 2, p.10]{kink}). For  uniqueness, consider two positive minimizers  $v_1$ and $v_2$. On the one hand, if $v_1(x_0)=v_2(x_0)$ for some $x_0$, then necessarily $E(v_1, [x_0,\infty))=E(v_2, [x_0,\infty))$. Setting
\[
\tilde v(x) = \begin{cases}
v_1(x) & x\leq x_0\\
v_2(x) & x\geq x_0
\end{cases}
\]
we see that $\tilde v$ is also a global minimizer of $E$ and hence a solution. This implies that $v'_1(x_0)=v'_2(x_0)$, and  by the uniqueness result for O.D.E., we deduce that $v_1\equiv v_2$. On the other hand, if $v_1>v_2> 0$, we obtain by \eqref{subaa} that $E(v_1)<E(v_2)$, which is impossible. This proves uniqueness.

(ii) By \eqref{digital}, it is clear that $v$ is convex on $[\xi,\infty)$, since $\mu\leq 0$ on $[\xi,\infty)$. Suppose by contradiction that $v'(x_0)\geq 0$, for some $x_0 \geq\xi$. Then, $v(x)\geq v(x_0)$ for $x \geq x_0$ which is impossible since $\lim_{+\infty}v=0$. Next, since $\max_{\R}v>0$ is attained at some point $x_m$, we deduce from \eqref{digital} that $\ve^2 v''(x_m)=v(x_m)(v^2(x_m)-\mu(x_m))\leq 0$. This implies that $v^2(x_m)\leq  \mu(x_m)\leq \max_\R \mu=\mu(\zeta)$.

(iii) For $x \geq x_0>\xi$, we have $0<v(x)\leq v(x_0)$, and $v''\geq -\frac{\mu(x_0)}{\epsilon^2} v$. By a standard comparison argument (cf. for instance \cite[Lemma A.1]{book}), we deduce (iii).

The properties of the odd minimizer $u$ follow from similar arguments. 
\end{proof}

\begin{figure}[h]
\includegraphics[width=.7\textwidth]{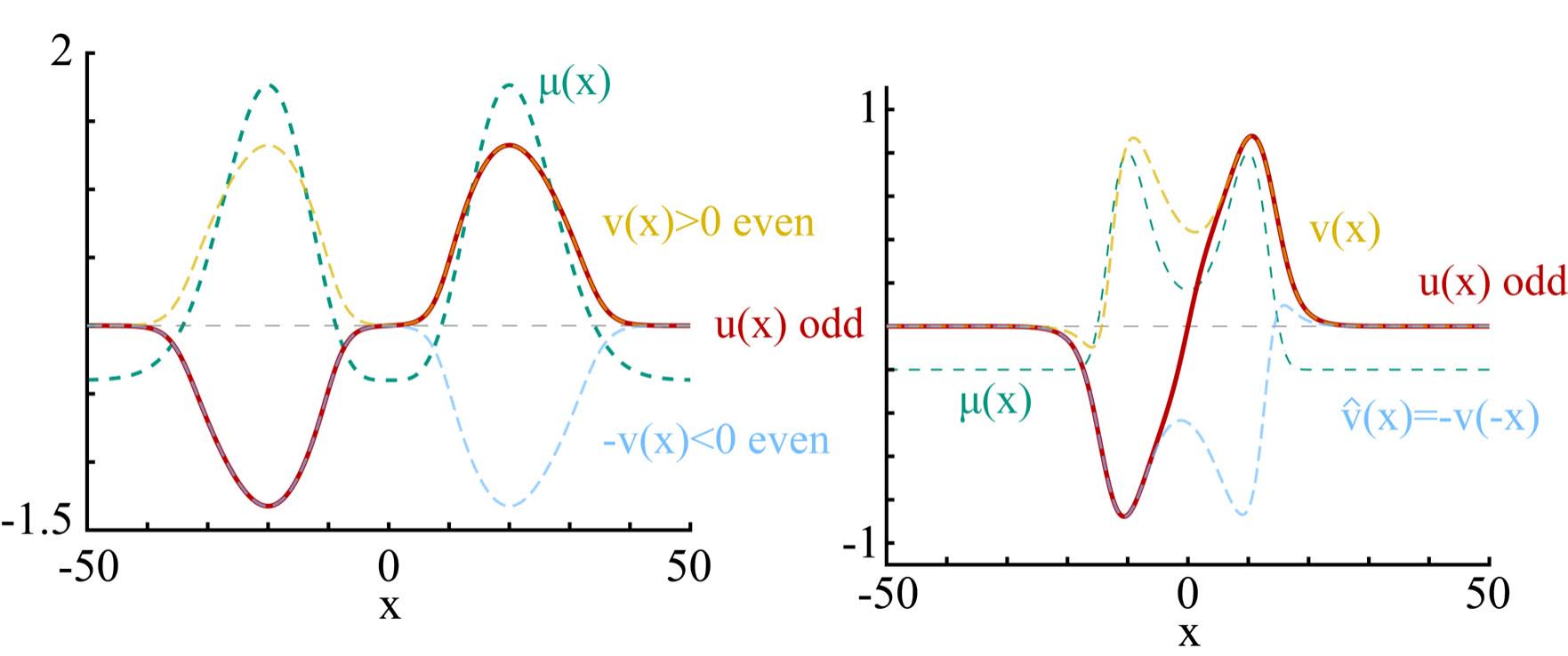}
\caption{An odd minimizer $u$ lying between two global minimizers. On the left under assumption (B), for $\alpha=0$. On the right under assumption (A), for $\alpha\in(0,\sqrt{2})$.}
\label{fig12b}
\end{figure}

\begin{proposition}\label{p4}
If $v_1$ and $v_2$ are two distinct global minimizers, then $v_1$ and $v_2$ do not intersect. Similarly, if $u_1$ and $u_2$ are two distinct odd minimizers, then $u_1$ and $u_2$ only intersect at the origin. Futhermore, a global minimizer $v$ either does not intersect an odd minimizer $u$ or coincides with it.
\end{proposition}
\begin{proof}
Suppose that $v_1(x_0)=v_2(x_0)$ for some $x_0$, then $$E(v_1,[x_0,\infty))= E(v_2,[x_0,\infty)).$$ As a consequence, $v_3$ defined by  $v_3(x)=v_1(x)$ for $x\leq x_0$ and $v_3(x)=v_2(x)$ for $x\geq x_0$ is also a global minimizer and a solution. Thus, $v'_1(x_0)=v'_2(x_0)$ and by the uniqueness result for O.D.E., we deduce that $v_1\equiv v_2$. Similarly, if $v(x_0)=u(x_0)$ for some $x_0$, which without loss of generality  can be assumed to be nonnegative, one can see that $E(v,[x_0,\infty))= E(u,[x_0,\infty))$, and $v'(x_0)=u'(x_0)$. As a consequence, $ u \equiv v$. In the same way, one can prove that two distinct odd minimizers only intersect at the origin.
\end{proof}

In Propositions \ref{p3} and \ref{p3ter} we study the zeros of the global and odd minimizers for arbitrary $\alpha>0$. 
Since the function $f$ changes sign $3$ times, one can show that $v$ and $u$ have at most $3$ zeros. Moreover, theses zeros are located in a neighborhood of the convex hull of the set $\{\mu\geq 0\}$.

\begin{proposition}\label{p3}(Zeros of the global and the odd minimizer)
\begin{itemize}
\item[(i)] Assuming that $\alpha>0$, the minimizer $v$ (resp. $u$) constructed in Theorem \ref{th1} has at most $3$ zeros, denoted by $\bar x_1<\bar x_2<\bar x_3$.
\item[(ii)] If $\epsilon>0$ and $\alpha>0$ remain in a bounded interval, there exists a constant $\delta>0$ such that when $\frac{\epsilon}{\alpha}<\delta$: the global minimizer $v$ has at least one zero, and in addition the zeros of $v$ or $u$ satisfy $|\bar x_i|\leq \xi+\mathcal O(\sqrt{\epsilon/\alpha})$. Finally, $v<0$ (resp $u<0$) on $(-\infty,\bar x_1)$, while $v>0$ (resp. $u>0$) on $(\bar x_k,\infty)$, where we have denoted by $\bar x_k$, the biggest zero of $v$ (resp. $u$). 
\end{itemize}
\end{proposition}

\begin{proof}
The proof is similar for $v$ and $u$. Without loss of generality, we examine only the case of $v$.

(i) \underline{Claim 1}: \emph{$v$ has at most one zero $y$ on the interval $(-\infty,-\zeta]$ (resp. $[\zeta,\infty)$). In that case, it holds $v<0$ on $(-\infty,y)$ (resp. $v>0$ on $(y,\infty)$), and $v>0$ on $(-\infty,-\zeta]\setminus (-\infty,y]$ (resp. $v<0$ on $[\zeta,\infty)\setminus [y,\infty)$). }

Assuming that $v$ has a zero $y\in(-\infty,-\zeta]$, we first show that $v(z)<0$, $\forall z<y$. Indeed, if $v(z)>0$ for some $z<y$, then one can see that $E(v ,(-\infty,y])>E(-|v|,(-\infty,y])$, which is a contradiction. Moreover, according to what precedes, a point $z<y$ such that $v(z)=0$, would be a local maximum of $v$. This is also excluded, since  we have $\epsilon v''(z)=-\alpha f(z)>0$ by \eqref{digital}. Thus we have proved that $v(y)=0$, with $y \leq -\zeta$, implies that $v(z)<0$, $\forall z<y$. Consequently, it is clear that $y$ is the only zero of $v$ in the interval $(-\infty,-\zeta]$. Furthermore, if $y<-\zeta$, then $v(z)>0$ for $z \in (y,-\zeta]$, since by \eqref{digital}, $v$ cannot have a local maximum at $y$.
This establishes Claim 1 on $(-\infty,-\zeta]$ (the proof is similar on $[\zeta,\infty)$).

\underline{Claim 2}: \emph{$v$ has at most two zeros $y_1$ and $y_2>y_1$ on the interval $[-\zeta,0]$ (resp. $[0,\zeta]$). In that case, it holds $v>0$ on $(y_1,y_2)\subset [-\zeta,0]$  (resp. $v<0$ on $(y_1,y_2)\subset [0,\zeta]$), and $v<0$ on $[-\zeta,0]\setminus [y_1,y_2]$ (resp. $v>0$ on $[0,\zeta]\setminus[y_1,y_2]$).}

To prove Claim 2, we proceed as in Claim 1. In the case of the interval $[-\zeta,0]$, we first show that $v\geq 0$ on $(y_1,y_2)$, by comparing $E(v ,[y_1,y_2])$ with $E(|v|,[y_1,y_2])$. Next, in view of \eqref{digital}, we exclude the possibility of a local minimum in $(-\zeta,0)$, equal to $0$. 

\underline{Claim 3}: \emph{If $v$ has two zeros in $[-\zeta,0]$ (resp. in $[0,\zeta]$), then $v$ has at most three zeros in $\R$.}

If $v$ has two zeros in $[-\zeta,0]$, then by Claim 2, we have $v(-\zeta)\leq 0$, as well as $v(0)\leq 0$. Consequently, Claim 1 implies that $v$ does not vanish on $(-\infty,-\zeta)$, while Claim 2, implies that $v$ has at most one zero in $(0,\zeta]$. If $v$ does not vanish on  $(0,\zeta)$, we are done by Claim 1. Otherwise, $v$ vanishes only once on $(0,\zeta)$, as well as on $(0,\zeta]$, at a point $z$ such that $v''(z)>0$, and $v<0$ on $(0,z)$ (cf. Claim 2 when $v(0)=0$). Consequently, $v(\zeta)>0$, and by Claim 1, $v$ does not vanish on $[\zeta,\infty)$. This establishes Claim 3 on $[-\zeta,0]$ (the proof is similar on $[0,\zeta]$). 

\underline{Claim 4}: \emph{If $v$ has none zero in $[-\zeta,0]$ (resp. in $[0,\zeta]$), then $v$ has at most three zeros in $\R$.}
   
Assume $v$ has none zero in $[-\zeta,0]$ (resp. in $[0,\zeta]$). If $v$ has two zeros in $[0,\zeta]$ (resp. in $[-\zeta,0]$), we are done by Claim 3. Otherwise, $v$ has at most one zero in $[0,\zeta]$ (resp. in $[-\zeta,0]$), and at most one zero in $(-\infty,-\zeta)$ as well as in $(\zeta,\infty)$. This proves Claim 4.

\underline{Claim 5}: \emph{If $v$ vanishes exactly once in $[-\zeta,0]$ and in $[0,\zeta]$, then $v$ has at most three zeros in $\R$.}

Indeed, if $v$ does not vanish on $(-\infty,-\zeta)$ or if $v(0)=0$, we are done by Claim 1. Otherwise,  $v$ vanishes once on $(-\infty,-\zeta)$ and $v(0)\neq 0$. Moreover, we have by Claim 1 that $v(-\zeta)>0$, while $v$ vanishes and changes sign at a point $y_1\in(-\zeta,0)$ (since $v''(y_1)<0$ by \eqref{digital}). In particular, we have $v(0)<0$.  Next, we distinguish two subcases. If $v(\zeta)=0$, we are done by Claim 1. Otherwise, $v$ vanishes and changes sign at a point $y_2\in(0,\zeta)$ (since $v''(y_2)>0$ by \eqref{digital}). Consequently, $v(\zeta)>0$, and Claim 1 implies that $v$ has none zero in $(\zeta,\infty)$.
This completes the proof of Claim 5, and Proposition \ref{p3} (i).

(ii) Suppose that $x_0<-\xi$ is such that $v(x_0)> 0$. Our first claim is that $v'(x_0)>0$. Indeed, suppose by contradiction that $v'(x_0)\leq 0$. Setting $$m:=\inf\{ x<x_0: \ v \geq 0\text{ on }[x,x_0]\},$$
one can see by \eqref{digital}, that $v$ is convex on the interval $(m,x_0]$, and thus $v \geq  v(x_0)$ on $(m,x_0]$. It follows that $m=-\infty$, which is a contradiction since $\lim_{-\infty}v=0$. This proves our claim. Now, let $M>0$ be the constant (cf. Proposition \ref{p0} (ii)), such that $|v_{\epsilon,\alpha}|\leq M$ when $\epsilon$ and $\alpha$ remain bounded, and let $m'=\min_{[-\xi-1,-\xi]}(-f)$. According to what precedes, we have $v'\geq 0$ and $v \geq 0$ on the interval $[x_0,-\xi]$, thus in view of \eqref{digital} we have $v''\geq -\frac{\alpha}{\epsilon}f$ on $[x_0,-\xi]$. In particular, for any  $x_0 \in [-\xi-1,-\xi)$ such that $v(x_0)> 0$, we obtain 
\begin{equation}\label{equzero}
M\geq v(-\xi)-v(x_0)\geq m' \frac{\alpha}{\epsilon} \frac{(\xi+x_0)^2}{2}.
\end{equation}
Finally, by taking $\frac{\epsilon}{\alpha}<\delta:=\frac{m'}{2M}$, we have $-\xi -\sqrt{\frac{\epsilon}{\alpha}\frac{2M}{m'}}>-\xi-1$, and since \eqref{equzero} does not hold for $-\xi-1 \leq x_0<-\xi -\sqrt{\frac{\epsilon}{\alpha}\frac{2M}{m'}}$, it follows that $v(x_0)\leq 0$, $\forall x_0 \in \big[-\xi-1 , -\xi -\sqrt{\frac{\epsilon}{\alpha}\frac{2M}{m'}}\big)$. By Claim 1 in the proof of (i), we conclude that  $v(x_0)<0$, $\forall x_0<-\xi -\sqrt{\frac{\epsilon}{\alpha}\frac{2M}{m'}}$. Similarly, we have $v(x_0)>0$, $\forall x_0>\xi +\sqrt{\frac{\epsilon}{\alpha}\frac{2M}{m'}}$, when $\frac{\epsilon}{\alpha}<\delta$. The proof of (ii) is now complete.
\end{proof}

When assumption (B) holds, we have the following more precise result:
\begin{proposition}\label{p3ter}(Zeros of the minimizers under assumption (B))
If assumption (B) holds, and if $\epsilon$ and $\alpha>0$ remain in a bounded interval, there exists a constant $\delta>0$ such that when $\frac{\epsilon}{\alpha}<\delta$, the minimizers $v$ and $u$ have exactly $3$ zeros. Moreover, $u$ and $v$ change sign at $\bar x_1$, $\bar x_2$, $\bar x_3$, and
\begin{itemize}
\item $\bar x_1 \in \big(-\xi -\mathcal{O }\big(\sqrt{\epsilon/\alpha}\big), -\xi'+\mathcal{O }\big(\sqrt{\epsilon/\alpha}\big)\big)$, 
\item $|\bar x_2|\leq\mathcal{O }\big(({\epsilon/\alpha})^\frac{\beta}{\beta+2}\big)$,
\item $\bar x_3 \in \big(\xi' -\mathcal{O }\big(\sqrt{\epsilon/\alpha}\big), \xi+\mathcal{O }\big(\sqrt{\epsilon/\alpha}\big)\big)$,
\end{itemize}
\end{proposition}

\begin{proof}
It is a direct consequence of Proposition \ref{p3}, and Lemma \ref{p3bbis} below.
\end{proof}
\begin{lemma}\label{p3bbis}
We assume that (B) holds. For $\epsilon$ and $\alpha$ belonging to a bounded interval, let $u_{\epsilon,\alpha}$ be a solution of \eqref{digital} converging to $0$ at $\pm\infty$. Then, $u_{\epsilon,\alpha}(x) <0$ for $0<\mathcal{O }\big(({\epsilon/\alpha})^\frac{\beta}{\beta+2}\big)<x<\xi' -\mathcal{O }\big(\sqrt{\epsilon/\alpha}\big)$, and similarly $u_{\epsilon,\alpha}(x) >0$ for $0<\mathcal{O }\big(({\epsilon/\alpha})^\frac{\beta}{\beta+2}\big)<-x<\xi' -\mathcal{O }\big(\sqrt{\epsilon/\alpha}\big)$,   $\epsilon/\alpha \ll 1$.
\end{lemma}

\begin{proof}
Let $M>0$ be the constant (cf. Proposition \ref{p0} (ii)), such that $|u_{\epsilon,\alpha}|\leq M$ when $\epsilon$ and $\alpha$ remain bounded.
For small $\delta>0$, we define $m_\delta=\min_{[\delta,\xi']}(-f)>\lambda \delta^\beta$, for a constant $\lambda>0$ (since $\liminf_{x\to 0}\frac{|f(x)|}{|x|^\beta}>0$), and set $m'=\min_{[\frac{\xi'}{2},\xi']}(-f)$. Let us prove our first claim: if $\frac{\epsilon}{\alpha}<\delta_1:=\frac{\lambda}{2M} \big(\frac{\xi'}{4}\big)^{\beta+2}$ and $x_0 \in \big(2\big(\frac{2M\epsilon}{\lambda \alpha}\big)^{\frac{1}{\beta+2}},\xi'\big]$ is such that $u(x_0)>0$, then $u'(x_0)>0$.
Suppose by contradiction that $u'(x_0)\leq 0$. Then, in view of \eqref{digital}, one can see that $u$ is positive and convex on the interval $I:=\big[\big(\frac{2M\epsilon}{\lambda \alpha}\big)^{\frac{1}{\beta+2}},x_0\big]$. In addition, $u''> \frac{\lambda \alpha}{\epsilon}\big(\frac{2M\epsilon}{\lambda \alpha}\big)^{\frac{\beta}{\beta+2}}$ on $I$, thus we obtain for $x \in I$:
\begin{equation}\label{equzerobis}
M\geq u(x)-u(x_0) >u'(x_0)(x-x_0)+\frac{\lambda \alpha}{2\epsilon} \Big(\frac{2M\epsilon}{\lambda \alpha}\Big)^{\frac{\beta}{\beta+2}}|x-x_0|^2,\nonumber
\end{equation} 
and taking $x =\big(\frac{2M\epsilon}{\lambda \alpha}\big)^{\frac{1}{\beta+2}} $, we reach a contradiction. This proves our first claim. To conclude, we are going to show that if $\frac{\epsilon}{\alpha}<\delta_2:=\min\big(\delta_1,\frac{m'}{2M} \big(\frac{\xi'}{2}\big)^2\big)$ and $x_0 \in \big(2\big(\frac{2M\epsilon}{\lambda \alpha}\big)^{\frac{1}{\beta+2}}, \xi'-\sqrt{\frac{2M\epsilon}{m' \alpha}}\big)$, then $u(x_0)< 0$. Suppose by contradiction that $u(x_0)>0$. In view of our first claim, it holds $u'(x_0)>0$. Moreover, $u$ is positive and convex on $[x_0,\xi']$, thus $u'$ is also positive on $[x_0,\xi']$. Setting $z=\xi'-\sqrt{\frac{2M\epsilon}{m' \alpha}}$, and $y=\max(x_0,\frac{\xi'}{2})$, we have  $\frac{\epsilon}{\alpha}<\delta_2\Rightarrow \frac{\xi'}{2}<z$, and $M\geq u(\xi')-u(y)\geq  \frac{m' \alpha}{2\epsilon} |\xi'-y|^2$, which is a contradiction. Thus we have proved that $u \leq 0$ on $J:=\big(2\big(\frac{2M\epsilon}{\lambda \alpha}\big)^{\frac{1}{\beta+2}}, \xi'-\sqrt{\frac{2M\epsilon}{m' \alpha}}\big)$, when $\frac{\epsilon}{\alpha}<\delta_2$. The strict inequality $u<0$ on $J$ follows from \eqref{digital}, since $u$ cannot have a local maximum equal to $0$ on this interval. Similarly, one can show that $u_{\epsilon,\alpha}(x) >0$ for $0<\mathcal{O }\big(({\epsilon/\alpha})^\frac{\beta}{\beta+2}\big)<-x<\xi' -\mathcal{O }\big(\sqrt{\epsilon/\alpha}\big)$, and $\epsilon/\alpha \ll 1$.
\end{proof}

\section{Upper bounds of the renormalized energy}
\label{sec:renormalized}

In the next Propositions, we compute an upper bound of the renormalized energy for the global and the odd minimizers, as $\epsilon \to 0$. These bounds are obtained in each case by choosing among all possible candidates, the test function with the smallest energy. It will appear in the proof of Theorem \ref{p10sh} below, that the test functions chosen in Propositions \ref{p5aa} and \ref{p5} are relevant, since they mimic the behaviour of minimizers. 

\begin{proposition}\label{p5aa}

Under hypothesis (A), we have for the global minimizers $v$ constructed in Theorem \ref{th1}:
\begin{equation}\label{qua22}
\limsup_{\epsilon\to 0, \alpha\to \alpha_0}\mathcal{E}(v_{\epsilon,\alpha})\leq 
\begin{cases}
0   &\text{for } 0\leq \alpha_0 \leq \sqrt{2}, \\
\frac{2}{3}(\sqrt{2}-\alpha_0)(\mu(\zeta))^{3/2}&\text{for } \sqrt{2}\leq \alpha_0\leq \alpha^{**} ,\\
\frac{2}{3}\big(2(\sqrt{2}-\alpha_0)(\mu(\zeta))^{3/2}+(\sqrt{2}+\alpha_0)(\mu(0))^{3/2}\big) &\text{for } \alpha_0 \geq \alpha^{**},
\end{cases}
\end{equation}
where $\alpha^{**}:=\sqrt{2}\frac{(\mu(\zeta))^{3/2}+(\mu(0))^{3/2}}{(\mu(\zeta))^{3/2}-(\mu(0))^{3/2}}$.
Similarly, we have for the odd minimizers $u$ constructed in Theorem \ref{th1}:
\begin{equation}\label{quab122odd}
\limsup_{\epsilon\to 0, \alpha\to \alpha_0}\mathcal{E}(u_{\epsilon,\alpha})\leq 
\begin{cases}
\frac{2}{3}(\sqrt{2}-\alpha_0)(\mu(0))^{3/2}&\text{for } 0\leq \alpha_0\leq \alpha_{\mathrm{odd}} ,\\
\frac{2}{3}\big(2(\sqrt{2}-\alpha_0)(\mu(\zeta))^{3/2}+(\sqrt{2}+\alpha_0)(\mu(0))^{3/2}\big)&\text{for } \alpha_0 \geq \alpha_{\mathrm{odd}},
\end{cases}
\end{equation}
where $\alpha_{\mathrm{odd}}:=\frac{\sqrt{2}(\mu(\zeta))^{3/2}}{(\mu(\zeta))^{3/2}-(\mu(0))^{3/2}} \in (\sqrt{2},\alpha^{**})$.
\end{proposition}

\begin{proof}[Proof of Proposition~\ref{p5aa}]
We consider the $C^1$ piecewise even function:
\begin{equation}
\phi(x)=
\begin{cases}
k_\epsilon \epsilon e^{-\frac{|x|-\xi}{\epsilon^2}}  &\text{for } |x|\geq \xi-\epsilon^2,\\
\sqrt{\mu(x)} &\text{for }  | x|\leq \xi -\epsilon^2, 
\end{cases}\nonumber
\end{equation} 
$$\text{with $k_\ve$ defined by $
k_\epsilon \epsilon e:=\sqrt{\mu(\xi- \epsilon^2)}\Longrightarrow k_\epsilon=\mathcal O(1).$}$$
Since $\phi \in H^1(\R)$, we have $\mathcal E(v)\leq \mathcal E(\phi)$, and moreover
\begin{align*}
\mathcal  E(\phi)=\int_{|x|\leq\xi-\epsilon^2}\frac{\epsilon}{2}\frac{ |\mu'|^2}{4\mu}+\int_{\xi-\epsilon^2\leq|x|\leq \xi}\frac{\mu^2}{4\epsilon}+ E(\phi, (-\infty,-\xi+\epsilon^2]\cup[\xi-\epsilon^2,\infty))=\mathcal O(\epsilon | \ln \epsilon |)+\mathcal O(\epsilon )+\mathcal O(\epsilon ). 
\end{align*}

Next, we repeat the previous computation by considering another $C^1$ piecewise function:
\begin{equation}
\eta(x)=
\begin{cases}
-k_\epsilon \epsilon e^{\frac{x+\xi}{\epsilon^2}}  &\text{for } x\leq -\xi+\epsilon^2,\\
-\sqrt{\mu(x)} &\text{for } -\xi+\epsilon^2\leq x\leq -\zeta-\zeta_\epsilon \epsilon, \\
l_\epsilon^- \tanh\Big(\frac{x+\zeta}{\epsilon}\sqrt{\frac{\mu(\zeta)}{2}}\Big) &\text{for }-\zeta-\zeta_\epsilon \epsilon \leq x \leq -\zeta, \\
l_\epsilon^+ \tanh\Big(\frac{x+\zeta}{\epsilon}\sqrt{\frac{\mu(\zeta)}{2}}\Big) &\text{for }-\zeta \leq x \leq -\zeta+\zeta_\epsilon \epsilon, \\
\sqrt{\mu(x)} &\text{for } -\zeta+ \zeta_\epsilon\epsilon\leq x\leq  \xi-\epsilon^2, \\
k_\epsilon \epsilon e^{-\frac{x-\xi}{\epsilon^2}}  &\text{for } x\geq \xi-\epsilon^2,
\end{cases}\nonumber
\end{equation} 
$$\text{with $\zeta_\epsilon:=- \ln \epsilon$, \  $k_\epsilon$ as above}, \text{ and } \ l^\pm_\epsilon \tanh\Big(\zeta_\epsilon\sqrt{\frac{\mu(\zeta)}{2}}\Big):=\sqrt{\mu(-\zeta\pm\zeta_\epsilon\epsilon)}\Longrightarrow \lim_{\epsilon \to 0}l^\pm_\epsilon=\sqrt{\mu(\zeta)}.$$ 
Clearly, we have $\mathcal E(v)\leq \mathcal E(\eta)$, and $\eta$ convergences a.e. to $\sqrt{\mu^+}$ (resp. $-\sqrt{\mu^+}$), when $x>-\zeta$ (resp. $x<-\zeta$). Moreover, we use the auxilliary function $\tilde \eta(s):=\sqrt{\mu(\zeta)}\tanh\Big(s\sqrt{\frac{\mu(\zeta)}{2}}\Big)$, which is the heteroclinic orbit corresponding to the potential $W(t):=\frac{(t^2-\mu(\zeta))^2}{4}$. It is known that $\int_\R\big(\frac{|\tilde \eta'|^2}{2}+W(\tilde\eta)\big)=\frac{2\sqrt{2}}{3}(\mu(\zeta))^{3/2}$. In view of what precedes, we obtain 
\begin{equation}\label{renorm3b3}
\mathcal E(\eta) \to \frac{2}{3}(\sqrt{2}-\alpha_0)(\mu(\zeta))^{3/2},
\end{equation}
as $\epsilon \to 0$, and $\alpha\to \alpha_0$, since $\mathcal E(\eta)$ is the sum of the following terms:
$$E(\eta, (-\infty,-\xi+\epsilon^2]\cup [\xi-\epsilon^2,\infty))=\mathcal O(\epsilon),$$ 
$$\int_{[-\xi+\epsilon^2,-\zeta-\zeta_\epsilon\epsilon]\cup[-\zeta+\zeta_\epsilon\epsilon,\xi-\epsilon^2]}\frac{\epsilon}{2}\frac{ |\mu'|^2}{4\mu}=\mathcal O(\epsilon | \ln \epsilon |), \ \int_{\xi-\epsilon^2\leq|x|\leq \xi}\frac{\mu^2}{4\epsilon}=\mathcal O(\epsilon), $$
$$\int_{0<x+\zeta< \zeta_\epsilon \epsilon}\frac{\epsilon|\eta'|^2}{2}+\frac{(\eta^2-\mu)^2}{4\epsilon}=\int_0^{\zeta_\epsilon}\big(\frac{(l_\epsilon^+)^2}{2\mu(\zeta)}|\tilde \eta'(s)|^2+\frac{(\frac{(l_\epsilon^+)^2}{\mu(\zeta)}(\tilde\eta(s))^2 -\mu(-\zeta+\epsilon s)\big)^2}{4}\big)ds\to \frac{\sqrt{2}}{3}(\mu(\zeta))^{3/2}, $$
$$\int_{-\zeta_\epsilon \epsilon<x+\zeta<0 }\frac{\epsilon|\eta'|^2}{2}+\frac{(\eta^2-\mu)^2}{4\epsilon}=\int_{-\zeta_\epsilon}^0\big(\frac{(l_\epsilon^-)^2}{2\mu(\zeta)}|\tilde \eta'(s)|^2+\frac{(\frac{(l_\epsilon^-)^2}{\mu(\zeta)}(\tilde\eta(s))^2 -\mu(-\zeta+\epsilon s)\big)^2}{4}\big)ds\to \frac{\sqrt{2}}{3}(\mu(\zeta))^{3/2}, $$
$$-\int_\R \alpha f\eta\to \int_{-\xi}^{-\zeta} \alpha_0 f\sqrt{\mu}-\int_{-\zeta}^\xi\alpha_0 f\sqrt{\mu}= -\int_{-\xi}^{-\zeta} \alpha_0 \frac{\mu'}{2}\sqrt{\mu}+\int_{-\zeta}^\xi\alpha_0 \frac{\mu'}{2}\sqrt{\mu}=-\frac{2\alpha_0}{3}(\mu(\zeta))^\frac{3}{2}.$$

Finally, we choose the following $C^1$ piecewise odd function:
\begin{equation}
\psi(x)=
\begin{cases}
\eta(x) &\text{for }  x\leq -\zeta_\epsilon\epsilon , \\
-l'_\epsilon \tanh\Big(\frac{x}{\epsilon}\sqrt{\frac{\mu(0)}{2}}\Big) &\text{for } |x|\leq \zeta_\epsilon \epsilon, \\
-\eta(-x) &\text{for } x \geq \zeta_\epsilon \epsilon.
\end{cases}\nonumber
\end{equation} 
with $\zeta_\epsilon$ as above, and $l'_\epsilon \tanh\Big(\zeta_\epsilon\sqrt{\frac{\mu(0)}{2}}\Big):=\sqrt{\mu(\zeta_\epsilon\epsilon)}\Longrightarrow 
\lim_{\epsilon \to 0}l'_\epsilon=\sqrt{\mu(0)}$. 
Since $\psi \in H^1_{\mathrm{odd}}(\R)$, we have $\mathcal E(u)\leq \mathcal E(\psi)$. By reproducing the previous computations, one can check that 
\begin{equation}\label{renorm3b4}
\mathcal E(\psi)\to \frac{2}{3}\big(2(\sqrt{2}-\alpha_0)(\mu(\zeta))^{3/2}+(\sqrt{2}+\alpha_0)(\mu(0))^{3/2}\big),
\end{equation}
as $\epsilon \to 0$, and $\alpha\to \alpha_0$. On the one hand, the quantity $\frac{2}{3}\big[2\sqrt{2}(\mu(\zeta))^{3/2}+\sqrt{2}(\mu(0))^{3/2}\big]$ is the contribution of the transitions at $0$ an $\pm\zeta$, while on the other hand $\lim_{\epsilon\to 0}\int_\R f\psi=\int_{\mu>0}|f|\sqrt{\mu}=\frac{2}{3}\big[2(\mu(\zeta))^{3/2}- (\mu(0))^{3/2}\big]$, since $\lim_{\epsilon\to 0} f\psi=|f|\sqrt{\mu^+}$ almost everywhere.

To derive \eqref{qua22}, we just notice that
\begin{equation*}
\min(\lim_{(\epsilon,\alpha)\to (0,\alpha_0)}\mathcal E(\phi),\lim_{(\epsilon,\alpha)\to (0,\alpha_0)}\mathcal E(\eta), \lim_{(\epsilon,\alpha)\to (0,\alpha_0)}\mathcal E(\psi))=
\begin{cases}\lim_{(\epsilon,\alpha)\to (0,\alpha_0)}\mathcal E(\phi)  &\text{for } 0\leq \alpha_0 \leq \sqrt{2}, \\
\lim_{(\epsilon,\alpha)\to (0,\alpha_0)}\mathcal E(\eta)  &\text{for } \sqrt{2}\leq \alpha_0\leq \alpha^{**} ,\\
\lim_{(\epsilon,\alpha)\to (0,\alpha_0)}\mathcal E(\psi) &\text{for } \alpha_0 \geq \alpha^{**}.
\end{cases}
\end{equation*}

To complete the proof of \eqref{quab122odd}, we need to construct another odd test function. We define a function $\chi$ equal to $\phi$ for $x\geq \zeta_\epsilon\epsilon$, and equal to $\psi$ for $0\leq x\leq \zeta_\epsilon\epsilon$. Then, we extend it by symmetry on all $\R$. As previously, one can see that $\mathcal E(u)\leq \mathcal E(\chi)$, and $\lim_{(\epsilon,\alpha)\to (0,\alpha_0)}\mathcal E(\chi)=\frac{2}{3}(\sqrt{2}-\alpha_0)(\mu(0))^{3/2}$, where $\frac{2\sqrt{2}}{3}(\mu(0))^{3/2}$ is the contribution of the transition at $0$, while  $\lim_{\epsilon\to 0}\int_\R f\chi=\frac{2}{3}(\mu(0))^{3/2}$ (since $\lim_{\epsilon\to 0}\chi(x)=\sqrt{\mu^+(x)}$ for a.e. $x\in(0,\infty)$). Furthermore, we have $\lim_{(\epsilon,\alpha)\to (0,\alpha_0)}\mathcal E(\chi)\leq\lim_{(\epsilon,\alpha)\to (0,\alpha_0)}\mathcal E(\psi)$, iff $\alpha_0\leq\alpha_{\mathrm{odd}}$.
\end{proof}

\begin{remark}
The functions $\phi$, and $\eta$ constructed in Proposition \ref{p5aa} mimic the behaviour of a global minimizer with a unique zero respectively at $-\xi $ and $-\zeta$ (cf. Figure \ref{fig2bb} left and center). On the one hand, both functions remain close to $\pm\sqrt{\mu^+}$, which is the a.e. limit of minimizers as $\epsilon\to 0$. Moreover, they have a corner layer type profile at the points $\pm\xi$ where $\sqrt{\mu}$ is not differentiable. Several scalings may be chosen so that their energy contribution on $(-\infty,\-\xi]\cup [\xi,\infty)$, tends to $0$. On the other hand, the function $\eta$ mimics the standard kink transition at $-\zeta$, while we take $\phi>0$, since the energy contribution of the shadow kink transition at $-\xi$ is of order $o(1)$.
\end{remark}
\begin{proposition}\label{p5}
Under hypothesis (B), we have for the minimizers $v$ and $u$ constructed in Theorem \ref{th1}:
\begin{equation}\label{qua5}
\limsup_{\epsilon\to 0, \alpha\to \alpha_0}\mathcal{E}(v_{\epsilon,\alpha})\leq \limsup_{\epsilon\to 0, \alpha\to \alpha_0}\mathcal{E}(u_{\epsilon,\alpha})\leq \frac{4\min(0,\sqrt{2}-\alpha_0)}{3}(\mu(\zeta))^{3/2}.
\end{equation}
\end{proposition}

\begin{proof}[Proof of Proposition~\ref{p5}]
We recall that $\xi'$ is the zero of $\mu$ in the interval $(0,\zeta)$, 
and consider the $C^1$ piecewise odd function:
\begin{equation}
\phi(x)=
\begin{cases}
-k_\epsilon \epsilon e^{\frac{x+\xi}{\epsilon^2}}  &\text{for } x\leq -\xi+\epsilon^2,\\
-\sqrt{\mu(x)} &\text{for } -\xi+\epsilon^2 \leq x\leq -\xi'-\epsilon^2 , \\
\kappa_\epsilon x &\text{for } |x|\leq \xi'+\epsilon^2, \\
\sqrt{\mu(x)} &\text{for } \xi'+\epsilon^2\leq x\leq \xi-\epsilon^2, \\
k_\epsilon \epsilon e^{-\frac{x-\xi}{\epsilon^2}}  &\text{for } x\geq \xi-\epsilon^2,
\end{cases}\nonumber
\end{equation} 
$$\text{with $k_\ve$ and $\kappa_\epsilon$ defined by }k_\epsilon \epsilon e:=\sqrt{\mu(\xi- \epsilon^2)}\Longrightarrow k_\epsilon=\mathcal O(1), \ \kappa_\epsilon(\xi'+\epsilon^2):=\sqrt{\mu(\xi'+\epsilon^2)}\Longrightarrow \kappa_\epsilon=\mathcal O(\epsilon).$$  
We have $\lim_{\epsilon\to 0}\phi(x)=\sqrt{\mu^+(x)}$ for a.e. $x\in(0,\infty)$, and it is clear that $\mathcal E(v)\leq \mathcal E(u)\leq \mathcal E(\phi)$, since $\phi \in H^1_{\mathrm{odd}}(\R)$. Moreover, it holds  $\lim_{(\epsilon,\alpha)\to (0,\alpha_0)}\mathcal E(\phi)=0$. Indeed, $\mathcal E(\phi)$ is the sum of the following terms:
$$\int_{\xi'+\epsilon^2\leq|x|\leq\xi-\epsilon^2}\frac{\epsilon}{2}\frac{ |\mu'|^2}{4\mu}=\mathcal O(\epsilon | \ln \epsilon |), \ \int_{\xi-\epsilon^2\leq|x|\leq \xi}\frac{\mu^2}{4\epsilon}=\mathcal O(\epsilon^2), \ \int_{\xi'\leq|x|\leq \xi'+\epsilon^2}\frac{\mu^2}{4\epsilon}=\mathcal O(\epsilon^2),$$
$$E(\phi, (-\infty,-\xi+\epsilon^2]\cup [\xi-\epsilon^2,\infty))=\mathcal O(\epsilon),  \ E(\phi, [-\xi'-\epsilon^2,\xi'+\epsilon^2])=\mathcal O(\epsilon) .$$ 
$$-\int_\R \alpha f\phi =-2\int_0^\infty \alpha f\phi  \to -2\int_0^\infty \alpha_0 f\sqrt{\mu^+}= \int_{\xi'}^\xi\alpha_0 \mu'\sqrt{\mu}=0.$$

To complete the proof of \eqref{qua5}, we need to construct another odd test function. We define the function $\psi$ by
\begin{equation}
\psi(x)=
\begin{cases}
-\kappa_\epsilon x &\text{for } 0\leq x\leq \xi'+\epsilon^2, \\
-\sqrt{\mu(x)} &\text{for } \xi'+\epsilon^2\leq x\leq \zeta-\zeta_\epsilon\epsilon,  \\
l_\epsilon^- \tanh\Big(\frac{x-\zeta}{\epsilon}\sqrt{\frac{\mu(\zeta)}{2}}\Big) &\text{for }\zeta-\zeta_\epsilon\epsilon\leq x\leq \zeta, \\
l_\epsilon^+ \tanh\Big(\frac{x-\zeta}{\epsilon}\sqrt{\frac{\mu(\zeta)}{2}}\Big) &\text{for } \zeta\leq x\leq \zeta+\zeta_\epsilon \epsilon, \\
\sqrt{\mu(x)} &\text{for } \zeta+\zeta_\epsilon \epsilon \leq x\leq \xi-\epsilon^2, \\
k_\epsilon \epsilon e^{-\frac{x-\xi}{\epsilon^2}}  &\text{for } x\geq \xi-\epsilon^2,
\end{cases}\nonumber
\end{equation}
and extend it by symmetry on all $\R$. The constants $k_\epsilon$, and $\kappa_\epsilon$ are the same as before, and we also take $\zeta_\epsilon:=- \ln \epsilon$, as well as $l_\epsilon^\pm \tanh\Big(\zeta_\epsilon\sqrt{\frac{\mu(\zeta)}{2}}\Big):=\sqrt{\mu(\zeta\pm\zeta_\epsilon\epsilon)}\Longrightarrow \lim_{\epsilon \to 0}l_\epsilon^\pm=\sqrt{\mu(\zeta)}$. 
We have $\lim_{\epsilon\to 0}f\psi=|f|\sqrt{\mu^+}$ for a.e. $x\in\R$, and it is clear that $\mathcal E(v)\leq \mathcal E(u)\leq \mathcal E(\psi)$, since $\psi \in H^1_{\mathrm{odd}}(\R)$. As previously, one can check that 
 $\lim_{(\epsilon,\alpha)\to (0,\alpha_0)}\mathcal E(\psi)=\frac{4\sqrt{2}}{3}(\mu(\zeta))^{3/2}-\int_{\mu>0}\alpha_0 |f|\sqrt{\mu}=\frac{4(\sqrt{2}-\alpha_0)}{3}(\mu(\zeta))^{3/2}$, where $\frac{4\sqrt{2}}{3}(\mu(\zeta))^{3/2}$ is the energy contribution of the transitions at $\pm\zeta$.
\end{proof}

\section{The structure of the kinks for the different regimes}\label{sec:structure}

\begin{theorem}\label{p10sh} 
Assume (A) and let $v_{\epsilon,\alpha}$ be a global minimizer. Then (up to the symmetry $\hat v(x)=-v(-x)$):
\begin{itemize}
\item[(i)] as $\epsilon \to 0$ and $\alpha \to \alpha_0 \in (0,\sqrt{2})$, $v_{\epsilon,\alpha}$ has a unique zero $\bar x_1(\epsilon,\alpha)$ converging to $-\xi$
\item[(ii)] as $\epsilon \to 0$ and $\alpha \to \alpha_0 \in (\sqrt{2},\alpha^{**})$ (cf. Proposition \ref{p5aa}), $v_{\epsilon,\alpha}$ has a unique zero $\bar x_1(\epsilon,\alpha)$ converging to $-\zeta$.
\item[(iii)] as $\epsilon \to 0$ and $\alpha \to \alpha_0 \in (\alpha^{**},\infty)$, $v_{\epsilon,\alpha}$ has three zeros $\bar x_1(\epsilon,\alpha) <\bar x_2(\epsilon,\alpha)<\bar x_3(\epsilon,\alpha)$ converging respectively to $-\zeta$, $0$ and $\zeta$.
\end{itemize}
Moreover, setting $ l_i=\lim_{\epsilon\to 0, \alpha\to \alpha_0} \bar x_i(\epsilon,\alpha)$, we have
\begin{equation}\label{sh11}
 \lim_{\epsilon\to 0, \alpha\to \alpha_0} v_{\epsilon,\alpha}(\bar x_i(\epsilon,\alpha)+s\epsilon)=\pm\sqrt{\mu(l_i)}\tanh(s\sqrt{\mu(l_i)/2})
\end{equation}
for the $C^1_{\mathrm{ loc}}(\R)$ convergence, and in case (i):
\begin{equation}\label{sh22}
\lim_{\epsilon\to 0, \alpha\to \alpha_0} v_{\epsilon,\alpha}(x_0+s\epsilon)=\sqrt{\mu^+(x_0)}, \ \forall x_0\in\R, 
\end{equation} 
in case (ii):
\begin{equation}\label{sh33}
\lim_{\epsilon\to 0, \alpha\to \alpha_0} v_{\epsilon,\alpha}(x_0+s\epsilon)=
\begin{cases}
-\sqrt{\mu^+(x_0)},  &\text{for } x_0\in (-\infty, -\zeta), \\
\sqrt{\mu^+(x_0)},  &\text{for } x_0\in (-\zeta,\infty),
\end{cases}
\end{equation} 
in case (iii):
\begin{equation}\label{sh44}
\lim_{\epsilon\to 0, \alpha\to \alpha_0} v_{\epsilon,\alpha}(x_0+s\epsilon)=
\begin{cases}
-\sqrt{\mu^+(x_0)},  &\text{for } x_0 \in (-\infty,-\zeta)\cup (0,\zeta), \\
\sqrt{\mu^+(x_0)},  &\text{for } x_0 \in (-\zeta,0)\cup(\zeta,\infty),  
\end{cases}
\end{equation} for the $C^1_{\mathrm{ loc}}(\R)$ convergence.
\end{theorem}

\begin{proof}
We consider a sequence $\epsilon_n\to 0$, and a sequence $\alpha_n \to \alpha_0$, and denote by $v_n$, a global minimizer $v_{\epsilon_n,\alpha_n}$ (up to the symmetry $\hat v(x)=-v(-x)$). 
It is known by Proposition \ref{p3}, that $v_n$ has at most $3$ zeros,  denoted by $\bar x_i(n)$, and that the set of limit points of these zeros is included in $[-\xi,\xi]$. 
Setting $ l_i:=\lim_{n \to \infty} \bar x_i(n)$ (up to subsequence), we will compute a lower bound of $\liminf_{n\to\infty}\mathcal{E}(v_n)$ for all possible values of
$-\xi\leq l_1\leq l_2\leq l_3\leq\xi$. Then, by comparing this lower bound with the upper bounds given in Proposition \ref{p5aa}, we will deduce the theorem.

Firstly, we rescale $v_n$, by setting $\tilde v_n(s):=v_n(\bar x_i(n)+s\epsilon_n)$, for some $i=1,2,3$, and compute the rescaled energy
\begin{equation}
\label{functres}
\tilde E(\tilde v)=\int_{\R}\Big(\frac{1}{2}|\tilde v'(s)|^2-\frac{1}{2}\mu(\bar x_i(n)+s\epsilon_n)\tilde v^2(s)+\frac{1}{4}|\tilde v|^4(s)-\epsilon_n\alpha_n f(\bar x_i(n)+s\epsilon_n)\tilde v(s)\Big)\dd s=E(v).
\end{equation}

In view of Proposition \ref{p0} (ii), the sequence $\tilde v_n$ is uniformly bounded, and since these functions solve the O.D.E.:
\begin{equation}\label{oderes}
\tilde v''_n(s)+\mu(\bar x_i(n)+s\epsilon_n) \tilde v_n(s)-\tilde v_n^3(s)+\epsilon_n \alpha_n f(\bar x_i(n)+s\epsilon_n)=0,\qquad  \forall s\in \R,
\end{equation}
it follows that the sequence $\tilde v_n$ is uniformly bounded, up to the second derivatives.
Thus, we can apply the theorem of Ascoli, via a diagonal argument, and show that for a subsequence, still called $\tilde v_n$,
$\tilde v_n$ converges in $C^1_{\mathrm{ loc}}(\R)$ to a function $\tilde V$.
To determine $\tilde V$, we consider test functions $\tilde \xi$ supported in a compact interval $J$, and  notice that $\tilde E_{\epsilon_n,\alpha_n}(\tilde v_n+\tilde \xi,J)\geq \tilde E_{\epsilon_n,\alpha_n}(\tilde v_n,J)$. At the limit we get
$\tilde E_{1}( \tilde V+\tilde\xi,J)\geq \tilde E_{1}(\tilde V,J)$, where $$\tilde E_{1}(\tilde v,J)=\int_{J}\Big(\frac{1}{2}|\tilde v'(s)|^2-\frac{1}{2}\mu(l_i)\tilde v^2(s)+\frac{1}{4}|\tilde v|^4(s)\Big)\dd s,$$
or equivalently $\tilde E_0( \tilde V+\tilde\xi,J)\geq \tilde E_{0}(\tilde V,J)$, where
\begin{multline}\label{gl}
\tilde E_{0}(\tilde v,J)=\int_{J}\Big(\frac{1}{2}|\tilde v'(s)|^2-\frac{1}{2}\mu(l_i)\tilde v^2(s)+\frac{1}{4}|\tilde v|^4(s)+\frac{(\mu(l_i))^2}{4}\Big)\dd s
=\int_{J}\Big(\frac{1}{2}|\tilde v'(s)|^2+\frac{1}{4}(\tilde v^2(s)-\mu(l_i))^2\Big)\dd s.
\end{multline}
Therefore, $\tilde V$ is a solution of the O.D.E.:
\begin{equation}\label{odeglb}
\tilde V''(s)-(\tilde V^2(s)-\mu(l_i))\tilde V(s)=0,
\end{equation}
which is minimal with respect to compactly supported perturbations. Moreover, since $\tilde V(0)=0$, we deduce that $\tilde V(s)= \pm\sqrt{\mu(l_i)}\tanh(s\sqrt{\mu(l_i)/2})$ (i.e. $\tilde V$ is a heteroclinic orbit when $\mu(l_i)>0$). 

At this stage, we distinguish the following cases.

a) Suppose first that for every $n$, $v_n$ has exactly one zero $\bar x_1(n)$ converging to $l_1$. Our claim is that 
\begin{equation}\label{claim12}
\liminf_{n \to\infty} \mathcal{E}(v_n)\geq \frac{2}{3}(\sqrt{2}-\alpha_0)(\mu(l_1))^{3/2}.
\end{equation}
We notice that if we rescale the sequence $v_n$ with respect to a point $x_0 \neq l_1$, by repeating the previous arguments, we will obtain the convergence of the sequence $v_n( x_0+s\epsilon_n)$, to a minimal solution of the O.D.E. $\tilde U''(s)-(\tilde U^2(s)-\mu(x_0))\tilde U(s)=0$, with constant sign, that is, to the constant $\pm \sqrt{\mu^+(x_0)}$. More precisely, we have
\begin{equation}\label{cvdsh}
\lim_{n\to\infty} v_n(x_0+s\epsilon)=
\begin{cases}
-\sqrt{\mu^+(x_0)},  &\text{for } -\infty<x_0< l_1, \\
\sqrt{\mu^+(x_0)},  &\text{for } l_1<x_0<\infty,
\end{cases}
\end{equation} for the $C^1_{\mathrm{ loc}}(\R)$ convergence (cf. Proposition \ref{p3} (Claim 1) for the sign of $v_n$). This allows us to compute the last term of the renormalized energy  \eqref{renorm}. 
In view of \eqref{cvdsh} and Proposition \ref{p0} (ii), we have by dominated convergence 
\begin{equation}\label{cvdoma}
\lim_{n\to\infty} -\int_\R\alpha_n f v_n =  \int_{-\xi}^{l_1}\alpha_0 f\sqrt{\mu}  -\int_{l_1}^{\xi}\alpha_0 f \sqrt{\mu} =-\frac{2}{3}\alpha_0(\mu(l_1))^{3/2}.
\end{equation}
On the other hand, if $\lambda_n>0$ is such that $\lim_{n\to\infty}\lambda_n=\infty$, we obtain by Fatou's Lemma:
\begin{align*}
\liminf_{n\to\infty}&\Big(\int_{\R}\frac{\epsilon}{2}|v_n'|^2+\int_{\mu>0}
\frac{(v_n^2-\mu)^2}{4\epsilon}+\int_{\mu<0}\frac{v_n^2(v_n^2-2\mu)}{4\epsilon}\Big)
\geq \liminf_{n\to\infty} \int_{\bar x_1(n)-\lambda_n\epsilon_n}^{\bar x_1(n)+\lambda_n\epsilon_n}\Big(\frac{\epsilon_n}{2}|v'_n|^2+\frac{(v_n^2-\mu^+)^2}{4\epsilon_n}\Big) \\
&\geq\liminf_{n\to\infty} \int_{-\lambda_n}^{\lambda_n}\Big(\frac{1}{2}|\tilde v'_n(s)|^2+\frac{(\tilde v^2_n(s)-\mu^+(\bar x_1(n)+s\epsilon_n))^2}{4}\Big)\dd s \\
&\geq \int_{\R}\Big(\frac{1}{2}|\tilde V'(s)|^2+\frac{(\tilde V^2_n(s)-\mu(l_1))^2}{4}\Big)\dd s =\frac{2\sqrt{2}}{3}(\mu(l_1))^{3/2},
\end{align*} 
which combined with \eqref{cvdoma} gives \eqref{claim12}.

b) Suppose now that for every $n$, $v_n$ has exactly two zeros $\bar x_1(n)<\bar x_2(n)$ converging respectively to $l_1\leq l_2$. 
In view of Proposition \ref{p3} (Claim 1), $v_n$ cannot change sign at both zeros, thus one of them is a local extremum. Suppose without loss of generality that for instance $\bar x_1(n)$ is a local maximum. Then,  it holds $\bar x_1(n)\notin (-\infty,-\zeta)$ (since otherwise we would have $ v_n''(\bar x_1(n))>0$ by \eqref{digital}), and $\bar x_1(n)\notin (\zeta,\infty)$ (cf. Proposition \ref{p3} (Claim 1)). That is, we have $\bar x_1(n)\in [-\zeta,\zeta]$, and $v_n(\bar x_1(n))=v_n'(\bar x_1(n))=0$. Consequently, the sequence $v_n(\bar x_1(n)+\epsilon s)$ converges to a minimal solution $\tilde V(s)$ of \eqref{odeglb} with $\mu(l_1)>0$, such that $\tilde V(0)=\tilde V'(0)=0$. Clearly, this is impossible.

c) Next, if for every $n$, $v_n$ has exactly three zeros $\bar x_1(n)<\bar x_2(n)<\bar x_3(n)$ converging respectively to $l_1< l_2<l_3$, we can establish as in case a) the lower bound:
\begin{equation}\label{claim14}
\liminf_{n \to\infty} \mathcal{E}(v_n)\geq \frac{2}{3}(\sqrt{2}-\alpha_0)(\mu(l_1))^{3/2}+\frac{2}{3}(\sqrt{2}+\alpha_0)(\mu(l_2))^{3/2}+\frac{2}{3}(\sqrt{2}-\alpha_0)(\mu(l_3))^{3/2},
\end{equation}
where the terms $\frac{2\sqrt{2}}{3}(\mu(l_3))^{3/2}$ are due to the transitions at $l_i$. On the other hand, one can see that the minimizers $v_n$ change sign at $\bar x_i(n)$ ($\forall i=1,2,3$), since the possibility of local extrema at $\bar x_i(n)$ is ruled out as in b) above.  Therefore, $\lim_{n\to\infty}v_n=-\sqrt{\mu^+}$ on $(-\infty,l_1)\cup(l_2,l_3)$, while $\lim_{n\to\infty}v_n=\sqrt{\mu^+}$ on $(l_1,l_2)\cup(l_3,\infty)$, and we also obtain
\begin{align*}
\lim_{n\to\infty} -\int_\R\alpha_n f v_n &=  \int_{-\xi}^{l_1}\alpha_0 f\sqrt{\mu}  -\int_{l_1}^{l_2}\alpha_0 f \sqrt{\mu} +\int_{l_2}^{l_3}\alpha_0 f\sqrt{\mu}  -\int_{l_3}^{\xi}\alpha_0 f \sqrt{\mu}\\
&=-\frac{2\alpha_0}{3}[(\mu(l_1))^{3/2}-(\mu(l_2))^{3/2}+(\mu(l_3))^{3/2}].
\end{align*}

d) Finally, we still assume that for every $n$, $v_n$ has three zeros $\bar x_1(n)<\bar x_2(n)<\bar x_3(n)$, but now we allow large inequalities for the corresponding limits: $l_1\leq l_2\leq l_3$. In this case, the lower bound \eqref{claim14} still holds, since we have $\lim_{n\to\infty} \frac{\bar x_{i+1(n)}-\bar x_i(n)}{\epsilon}=\infty$, when $l_i=l_{i+1}\neq\pm\xi$ ($i=1,2$). Indeed, if $\bar x_{i+1(n)}-\bar x_i(n)=o(\epsilon)$, there exists $y_n\in(\bar x_i(n),\bar x_{i+1(n)})$, such that $v_n'(y_n)=0$. Consequently, the sequence $\tilde v_n(s):=v_n(\bar x_i(n)+\epsilon s)$ satisfies $\lim_{n\to\infty} \tilde v_n'(0)=0$, and converges to a minimal solution $\tilde V(s)$ of \eqref{odeglb}, such that $\tilde V(0)=\tilde V'(0)=0$. This is excluded, as in b) above. On the other hand, if $\lim_{n\to\infty}\frac{\bar x_{i+1(n)}-\bar x_i(n)}{\epsilon}=\rho\in(0,\infty)$, we would obtain that $\tilde V=\lim_{n\to\infty}\tilde v_n$ is a minimal solution of \eqref{odeglb} vanishing at $0$ and $\rho$, a contradiction.

To conclude, we are going to determine in which of the cases above, the upper bounds \eqref{qua22} are satisfied. 

1) Clearly, when $0<\alpha_0\leq\sqrt{2}$, cases c) and d) above are excluded, since they imply a positive renormalized energy. The reason is that in case c) we have $|l_2|<\xi$, while in case d) the condition $l_i=l_{i+1}$ may only occur when $l_i=l_{i+1}\in[-\zeta,\zeta]$ (cf. Proposition \ref{p3} (Claim 1)). Thus, for $0<\alpha_0\leq\sqrt{2}$, $v_{\epsilon,\alpha}$ has a unique 
zero $\bar x_{1}(\epsilon,\alpha)$, as $\epsilon\to 0$. Moreover, if $0<\alpha_0<\sqrt{2}$, we have $\lim_{n\to \infty}\bar x_{1}(\epsilon,\alpha)=l_1=\pm\xi$, since only for these values the lower bound in \eqref{claim12} vanishes\footnote{For the critical value $\alpha_0=\sqrt{2}$, we know that the zero of $v_{\epsilon,\alpha}$ is unique, but \eqref{claim12} does not allow us to locate it.}.

2) When $\alpha_0>\sqrt{2}$, we consider the function
$G(l_1,l_2,l_3):=(\sqrt{2}-\alpha_0)(\mu(l_1))^{3/2}+(\sqrt{2}+\alpha_0)(\mu(l_2))^{3/2}+(\sqrt{2}-\alpha_0)(\mu(l_3))^{3/2}$
defined on the set $S=\{(l_1,l_2,l_3): -\xi<l_1< l_2< l_3<\xi\}$. We notice that 
\begin{equation}\label{eq123}
\min_{\partial S}G=(\sqrt{2}-\alpha_0)(
\mu(\zeta))^{3/2}
\begin{cases}
<G(l_1,l_2,l_3), \ \forall (l_1,l_2,l_3)\in S  &\text{ when } \sqrt{2}<\alpha_0<\alpha^{**}, \\
=G(-\zeta,0,\zeta) &\text{ when } \alpha_0=\alpha^{**},\\
>\min_S G=G(-\zeta,0,\zeta) &\text{ when } \alpha_0>\alpha^{**}.
\end{cases}
\end{equation} 
Thus, when $  \sqrt{2}<\alpha_0<\alpha^{**}$, the global minimizer $v_{\epsilon, \alpha}$ has a unique zero converging either to $\zeta$ or to $-\zeta$. Indeed, case c) is ruled out by the first inequality in \eqref{eq123}, while the equality of two limits $l_i=l_{i+1}$ implies that $l_i=l_{i+1}\in[-\zeta,\zeta]$ (cf. Proposition \ref{p3} (Claim 1)), as well as $(\sqrt{2}-\alpha_0)
\mu(\zeta))^{3/2}
<G(l_1,l_2,l_3)$.  
On the other hand, the equality holding in \eqref{eq123} when $\alpha_0=\alpha^{**}$, does not allow us to determine the number and location of zeros. 
Finally, when $\alpha_0>\alpha^{**}$, since the point $(-\zeta,0,\zeta)$ is the unique global minimum of $G$ in $\overline S$, it follows that the global minimizer $v_{\epsilon, \alpha}$ has exactly three zeros converging respectively to $-\zeta$, $0$ and $\zeta$. We also point out that the limits in \eqref{sh22}, \eqref{sh33}, \eqref{sh44}, are established as in \eqref{cvdsh}.
\end{proof}
Similarly, one can prove the following theorems that apply in case (A) to the odd minimizer, and in case (B) to the global and the odd minimizers:
\begin{theorem}\label{p10shbis} 
Assume (A) and let $u_{\epsilon,\alpha}$ be odd minimizers. Then,
\begin{itemize}
\item[(i)] as $\epsilon \to 0$ and $\alpha \to \alpha_0 \in (0,\alpha_{\mathrm{odd}})$, $u_{\epsilon,\alpha}$ has a unique zero located at the origin.
\item[(ii)] as $\epsilon \to 0$ and $\alpha \to \alpha_0 \in (\alpha_{\mathrm{odd}},\infty)$, $v_{\epsilon,\alpha}$ has two zeros converging to $\pm\zeta$, in addition to the one located at the origin.
\end{itemize}
\end{theorem}

\begin{proof}
Now, we compare the upper bounds \eqref{quab122odd}, with the lower bounds: 
\begin{itemize}
\item $\liminf_{(\epsilon,\alpha)\to (0, \alpha_0)} \mathcal{E}(u_{\epsilon,\alpha})\geq \frac{2}{3}(\sqrt{2}-\alpha_0)(\mu(0))^{3/2}$, induced by exactly one transition at $0$ (cf. \eqref{claim12} with $l_1=0$),
\item $\liminf_{(\epsilon,\alpha)\to (0, \alpha_0)} \mathcal{E}(u_{\epsilon,\alpha})\geq 
\frac{4}{3}(\sqrt{2}-\alpha_0)(\mu(l))^{3/2}+\frac{2}{3}(\sqrt{2}+\alpha_0)(\mu(0))^{3/2}$, induced by transitions at $0$ and $\pm l$ (cf. \eqref{claim14} with $l_1=-l$, $l_2=0$, $l_3=l$).
\end{itemize}

\end{proof}

\begin{theorem}\label{p10shter} 
Assume (B), and let $v_{\epsilon,\alpha}$ and $u_{\epsilon,\alpha}$ be respectively global, and odd minimizers. Then,
as $\epsilon \to 0$ and $\alpha \to \alpha_0>0$, $v_{\epsilon,\alpha}$ and $u_{\epsilon,\alpha}$ have three zeros $\bar x_1(\epsilon,\alpha)<\bar x_2(\epsilon,\alpha)<\bar x_3(\epsilon,\alpha)$ such that
\begin{itemize}
\item[(i)] as $\epsilon \to 0$ and $\alpha \to \alpha_0 \in (0,\sqrt{2})$, $\bar x_1(\epsilon,\alpha) \to -\xi \text{ or } -\xi'$, $\bar x_2(\epsilon,\alpha)\to 0$, and $\bar x_3(\epsilon,\alpha)\to \xi' \text{ or } \xi$. 
\item[(ii)] as $\epsilon \to 0$ and $\alpha \to \alpha_0 \in (\sqrt{2},\infty)$, $\bar x_1(\epsilon,\alpha) \to -\zeta$, $\bar x_2(\epsilon,\alpha)\to 0$, and $\bar x_3(\epsilon,\alpha)\to \zeta$. 
\end{itemize}
\end{theorem}

\begin{proof}
Let $\epsilon_n\to 0$, let $\alpha_n \to \alpha_0$, and denote by $v_n$, a global minimizer $v_{\epsilon_n,\alpha_n}$ (up to the symmetry $\hat v(x)=-v(-x)$). We know by Proposition \ref{p3ter}, that $v_n$ has exactly $3$ zeros, denoted by $\bar x_1(n)<\bar x_2(n)<\bar x_3(n))$. Moreover, $v_n$ changes sign at $\bar x_i(n)$ ($\forall i=1,2,3$), and setting $ l_i=\lim_{n \to \infty} \bar x_i(n)$ (up to subsequence), we have $l_1\in[-\xi,-\xi']$, $l_2=0$, and $l_3\in[\xi',\xi]$. 
One can see as in the proof of Theorem \ref{p10sh}, that $\lim_{n\to\infty}v_n=-\sqrt{\mu^+}$ on $(-\infty,l_1)\cup(l_2,0)$, while $\lim_{n\to\infty}v_n=\sqrt{\mu^+}$ on $(l_1,0)\cup(l_3,\infty)$. Therefore, we obtain
\begin{align*}
\lim_{n\to\infty} -\int_\R\alpha_n f v_n &=  \int_{-\xi}^{l_1}\alpha_0 f\sqrt{\mu}  -\int_{l_1}^{-\xi'}\alpha_0 f \sqrt{\mu} +\int_{\xi'}^{l_3}\alpha_0 f\sqrt{\mu}  -\int_{l_3}^{\xi}\alpha_0 f \sqrt{\mu}\\
&=-\frac{2\alpha_0}{3}[(\mu(l_1))^{3/2}+(\mu(l_3))^{3/2}].
\end{align*}
To conclude, we compare the upper bound \eqref{qua5} with the lower bound:
\begin{equation}\label{claim14bb}
\liminf_{n \to\infty} \mathcal{E}(v_n)\geq \frac{2}{3}(\sqrt{2}-\alpha_0)[(\mu(l_1))^{3/2}+(\mu(l_3))^{3/2}],
\end{equation}
where the terms $\frac{2\sqrt{2}}{3}(\mu(l_i))^{3/2}$ are due to the transitions at $l_1$ and $l_3$. In the case of the odd minimizer, we proceed in a similar way by taking $l_1=-l_3$.
\end{proof}

\begin{remark}
From the proofs of Theorems \ref{p10sh}, \ref{p10shbis} and \ref{p10shter}, it follows that \eqref{qua22}, \eqref{quab122odd} and \eqref{qua5} hold if we replace the inequalities by equalities, and take the limit of the renormalized energy, instead of the $\limsup$. However, the knowledge of this limit is not sufficient to determine the exact convergence of $\bar x_1(\epsilon,\alpha)$, and $\bar x_3(\epsilon,\alpha)$, in Theorem \ref{p10shter} (i). Smaller terms need to be computed in the expansion of the renormalized energy, and certainly the limits of $\bar x_1(\epsilon,\alpha)$ and $\bar x_3(\epsilon,\alpha)$, also depend on the values of $\mu'$ at $\xi$ and $\xi'$.
\end{remark}

The structure of the global minimizers and the odd minimizers follows from the convergence of their zeros as established in Theorems \ref{p10sh}, \ref{p10shbis} and \ref{p10shter}.
\begin{itemize}
\item[(1)] When the zero of the minimizer goes to $\pm \xi$ (or $\pm\xi'$) (the points where $\mu$ vanishes), the minimizer has 
the profile of a shadow kink. This means that
the appropriate rescaling to describe its behaviour near $\pm \xi$ (or $\pm\xi'$) is 
$\tilde v_\epsilon(s):=\frac{v_\epsilon(\xi+s\epsilon^{2/3})}{\epsilon^{1/3}}$. 
We know from \cite[Theorem 1.2]{kink} that as $\epsilon \to 0$, $\tilde v_\epsilon$ converges in $C^1_{\mathrm{ loc}}(\R)$ to a minimal solution $\tilde V$ of the Painlev\'{e} O.D.E.:
\begin{equation}\label{oderes4-0x}
\tilde V''(s)+\mu_1 s \tilde V(s)-\tilde V^3(s) + \alpha f(\xi)=0, \qquad \forall s\in \R,
\end{equation}
where $\mu_1:=\mu'(\xi)$.
\item[(2)] When the zero $\bar x(\epsilon)$ of the minimizer has a limit $l$ such that $\mu(l)>0$, the minimizer has the profile of a standard kink. 
We rescale it as we did in the proof of Theorem \ref{p10sh}, by setting
$\tilde v_\epsilon(s):=v_\epsilon(\bar x(\epsilon)+s\epsilon)$, and obtain the convergence of $\tilde v_\epsilon$ to a hyperbolic tangent 
(cf. \eqref{sh11}).
\item[(3)] Under assumption (B), when the zero of the minimizer goes to $0$, the minimizer has the profile of a giant kink. 
The appropriate rescaling is given by $\tilde v_{\epsilon,\alpha}(s):=\frac{v_{\epsilon,\alpha}(x_0+\epsilon s)}{\epsilon}$, $\forall x_0\in (-\xi',\xi')$, and the rescaled functions $\tilde v_{\epsilon,\alpha}$ converge in $C^1_{\mathrm{ loc}}(\R)$ to a constant:
\begin{equation}
\lim_{\epsilon\to 0, \alpha\to \alpha_0} \frac{v_{\epsilon,\alpha}(x_0+s\epsilon)}{\epsilon}=-\frac{\alpha_0 f(x_0)}{\mu(x_0)}.
\end{equation}
We also have the uniform convergence $\lim_{\epsilon\to 0, \alpha\to \alpha_0} \frac{v_{\epsilon,\alpha}(x)}{\epsilon}=-\frac{\alpha_0 f(x)}{\mu(x)}$, on compact intervals contained in $(-\xi',\xi')$. For these results we refer to Proposition \ref{pgia}.
\end{itemize}
To sum up, we can say that under assumption (A), and as $\epsilon\to 0$, the global minimizer $v_{\epsilon,\alpha}$ is composed of:
\begin{itemize}
\item[(i)] one shadow kink when $\alpha \in (0,\sqrt{2})$,
\item[(ii)] one standard kink when $\alpha \in(\sqrt{2},\alpha^{**})$, 
\item[(iii)] three standard kinks when $\alpha \in (\alpha^{**},\infty)$.
\end{itemize}
Similarly, under assumption (A), and as $\epsilon\to 0$, the odd minimizer $u_{\epsilon,\alpha}$ is composed of:
\begin{itemize}
\item[(i)] one standard kink when $\alpha  \in (0,\alpha_{\mathrm{odd}})$,
\item[(ii)] three standard kinks when $\alpha  \in (\alpha_{\mathrm{odd}},\infty)$.\end{itemize}
Finally, under assumption (B), and as $\epsilon \to 0$, the global minimizer $v_{\epsilon,\alpha}$ and the odd minimizer $u_{\epsilon,\alpha}$ are composed of:
\begin{itemize}
\item[(i)] two shadow kinks and a giant kink when $\alpha \in (0,\sqrt{2})$,
\item[(ii)] two standard kinks and a giant kink when $\alpha \in (\sqrt{2},\infty)$.
\end{itemize}

Before closing this section, we also mention two regimes for which the global minimizer is odd and unique.

\begin{theorem}\label{p4c}
Let $\epsilon_0 :=\sup_{ \phi \in H^1(\R), \phi \neq 0}\frac{(\int_{\R}\mu(x)\phi^2(x)\dd x)^{1/2}}{(\int_{\R}|\phi'|^2)^{1/2}}$, and  
$\epsilon_1 :=\sup_{ \phi \in H^1_{\mathrm{odd}}(\R), \phi \neq 0}\frac{(\int_{\R}\mu(x)\phi^2(x)\dd x)^{1/2}}
{(\int_{\R}|\phi'|^2)^{1/2}}$. Then, we have $0< \epsilon_1\leq \epsilon_0<\infty$. Moreover, if $\epsilon \geq \epsilon_0$, then the global minimizer $v$ is odd, i.e. it coincides with the odd minimizer $u$. It is also the unique solution of \eqref{digital} in $H^1(\R)$. In particular, when $\alpha=0$, then
\begin{itemize}
\item the global minimizer $v$ is identically $0$ if and only if $\epsilon\geq \epsilon_0$,
\item the odd minimizer $u$ is identically $0$ if and only if $\epsilon\geq\epsilon_1$.
\end{itemize}
\end{theorem}

\begin{proof}
It follows from \eqref{Poincare} that $0< \epsilon_1\leq \epsilon_0<\infty$. Moreover, \eqref{subbb} implies that for $\epsilon\geq \epsilon_0$, every solution $u \in H^1(\R)$ of \eqref{digital} satisfies $E(u+\phi)-E(u)\geq \int_{\R}\big(\frac{(u\phi)^2}{2\epsilon}+\frac{(2u+\phi)^2\phi^2}{4\epsilon}\big)>0$, for $\phi \in H^1(\R)\setminus\{0\}$. Thus, there exists a unique global minimizer which coincides with the odd minimizer, and it is the only solution of \eqref{digital} in $H^1(\R)$. If $\alpha=0$, and $\epsilon \geq \epsilon_0$, it is clear that the only solution is zero. On the other hand, if $\alpha=0$ and $\epsilon < \epsilon_0$, then there exists $\phi \in H^1(\R)$ such that $\int_{\R}\Big(\frac{\epsilon}{2}|\phi'|^2-\frac{1}{2\epsilon}\mu(x)\phi^2\Big)<0$. Therefore, for $\lambda>0$ small enough, we have $E(\lambda\phi)<0$. This implies that the global minimizer $v$ is distinct from $0$. The corresponding statement for the standard kink follows by a similar argument.
\end{proof}

\begin{theorem}\label{p4d}(cf. \cite[Theorem 1.3 (ii)]{symb} for a similar result)
Given $\epsilon>0$, there exists $A>0$ such that for every $\alpha>A$, the global minimizer $v$ is odd, i.e. it coincides with the odd minimizer $u$. It is also the unique solution of \eqref{digital} in $H^1(\R)$.
\end{theorem}

\begin{proof}
Assume $u_{\epsilon,\alpha}$ and $\hat u_{\epsilon,\alpha}$ are two solutions of \eqref{digital} in $H^1(\R)$. Setting, $\mu_\alpha:=\mu-u^2_{\epsilon,\alpha}$, and $\hat \mu_\alpha:=\mu-\hat u^2_{\epsilon,\alpha}$, we know by Proposition \ref{psta} that for $\epsilon>0$ fixed, and for every $x \notin\{0,\pm\zeta\}$, it holds $\lim_{\alpha\to 0}\mu_\alpha(x)=\lim_{\alpha\to 0}\hat \mu_\alpha(x)=-\infty$. Thus, by \eqref{Poincarebis}, there exists $A>0$, such that for every $\alpha>A$, we have (cf. \eqref{subbb}):
$$\int_{\R}\Big(\frac{\epsilon}{2}|\phi'|^2+\frac{( u_{\epsilon,\alpha}^2-\mu)}{2\epsilon}\phi^2\Big)>0\Longrightarrow E( u_{\epsilon,\alpha}+\phi)>E( u_{\epsilon,\alpha}), \ \forall \phi \in H^1(\R)\setminus\{0\}.$$  
Similarly, we obtain that
$E(\hat  u_{\epsilon,\alpha}+\phi)>E( \hat u_{\epsilon,\alpha})$, $\forall \phi \in H^1(\R)\setminus\{0\}$, for $\alpha>\tilde A$ large enough. Therefore, $u_{\epsilon,\alpha}\equiv \hat u_{\epsilon,\alpha}$, for $\alpha>\max(A,\hat A)$, which proves the uniqueness of the solution.
\end{proof}

\section{The case of a periodic function $\mu$}
\label{sec:per}

It is instructive to describe in this section the case of a periodic function $\mu$. We shall mention here the corresponding results, and just point out the new elements that are not trivial when adjusting the proofs. In what follows we assume that:
\begin{align}
\label{hyp3b}
\left\{
\begin{aligned}
 &\text{$\mu \in C^1(\R) $ is even and $T$-periodic (for some period  $T>0$), with $\mu'<0$ on $(0,T/2)$,}\\
 &\text{$f:=-\frac{\mu'}{2}$, with $\liminf_{x\to \frac{T}{2}}\frac{|f(x)|}{|x -\frac{T}{2}|^\beta}>0$, for some $\beta>0$, }\\
\end{aligned}
\right.
\end{align}
and suppose that one of the following is true:
\begin{itemize}
 \item[(A)] $\mu$ remains positive on the interval $[0,T/2]$,
 \item[(B)] $\mu$ changes sign on the interval $[0,T/2]$: we denote by $\xi$ its unique zero in $(0,T/2)$.
\end{itemize}
By considering the energy functional restricted to one period:
\begin{equation}
\label{funct 0per}
E(u)=\int_{0}^T\frac{\epsilon}{2}|u'|^2-\frac{1}{2\epsilon}\mu(x)u^2+\frac{1}{4\epsilon}|u|^4-\alpha f(x)u,
\end{equation}
we first obtain the existence of minimizers.
\begin{theorem}\label{th1per}
Assuming \eqref{hyp3b}, there exists $v \in H^1_{\mathrm{per}}([0,T]):=\{u \in H^1([0,T]) : u(0)=u(T) \}$ such that $E(v)=\min_{H^1_{\mathrm{per}}([0,T])} E$. Moreover, $v$ can be extended periodically on all $\R$ to define a $C^2(\R)$ classical solution of O.D.E. \eqref{digital}, that we still call $v$.
There also exists $u \in H^1_0([0,T/2])$ such that $E(u)=\min_{H^1_0([0,T/2])} E$, where here we consider the energy functional restricted to the interval $[0,T/2]$. Similarly, $u$ can be extended on all $\R$, to define an odd, and $T$-periodic, classical solution of O.D.E. \eqref{digital}, that we still call $u$.  
\end{theorem}

In the sequel, we will always denote by $v$ or $v_{\epsilon,\alpha}$ the global minimizer, and by $u$ or $u_{\epsilon,\alpha}$, the odd one.

\begin{proposition}\label{p0per}(Properties of solutions)
\begin{itemize}
\item[(i)] A $T$-periodic solution $u$ of \eqref{digital} satifies $E(u,[0,T])=\int_{0}^T\big(-\frac{1}{4\epsilon}|u|^4-\frac{\alpha}{2} f(x)u\big)$. If in addition $u$ is odd, then $E(u,[0,T/2])=\int_0^{T/2}\big(-\frac{1}{4\epsilon}|u|^4-\frac{\alpha}{2} f(x)u\big)$.
\item[(ii)] For $\epsilon$ and $\alpha$ belonging to a bounded interval, let $u_{\epsilon,\alpha}$ be a $T$-periodic solution of \eqref{digital}. Then, the solutions $u_{\epsilon,\alpha}$ are uniformly bounded by a constant $M$.
\end{itemize}
\end{proposition}

\begin{proposition}\label{p1per}(Properties of the global minimizer $v$ when $\alpha=0$)

When $\alpha=0$, let $v$ be the global minimizer constructed in Theorem \ref{th1per}. Then, 
\begin{itemize}
\item[(i)] $v$ has constant sign or is identically zero (in particular $v\not\equiv 0$, for $0<\epsilon \ll 1$).
\item[(ii)]$v$ is even, 
\item[(iii)] $v$ is unique up to change of $v$ by $-v$,
\item[(iv)] $v'(x)<0$, $\forall x \in (0,T/2)$, and $v(0)\leq\sqrt{ \mu(0)}$.
\end{itemize}
\end{proposition}

\begin{proposition}\label{p2per}(Properties of the odd minimizer $u$ for arbitrary $\alpha\geq 0$) 

When $\alpha=0$, the odd minimizer $u$ constructed in Theorem \ref{th1per} has constant sign on the interval $(0,T/2)$ or is identically zero (in particular $v\not\equiv 0$, for $0<\epsilon \ll 1$). Furthermore, it is unique up to change of $u$ by $-u$. On the other hand, when $\alpha>0$, the odd minimizer $u$ is positive on the interval $(0,T/2)$ and unique.
\end{proposition}

\begin{proposition}\label{p4per}(Superposition principle)

If $v_1$ and $v_2$ are two distinct global minimizers, then $v_1$ and $v_2$ do not intersect. Futhermore, a global minimizer $v$ either does not intersect an odd minimizer $u$ or coincides with it.
\end{proposition}
\begin{proof}
Assume first that the global minimizers $v_1$ and $v_2$ intersect twice within a period, at the points $0\leq x_1<x_2<T$. Then, it holds $E(v_1, [x_1,x_2])=E(v_2, [x_1,x_2])$, and setting $\tilde v(x)=v_1(x)$ for $x\in [x_1,x_2]$, while $\tilde v(x)=v_2(x)$ for $x\in[0,T]\setminus (x_1,x_2)$, it follows that  $\tilde v$ is a global minimizer and a solution. Consequently, we have $v_1\equiv v_2$. Now, if $v_1$ and $v_2$ intersect only once on $[0,T)$ at a point $x_0$, it holds for instance $v_1 \leq v_2$ on $\R$, and $v_1(x_0)=v_2(x_0)$, which implies that $v'_1(x_0)=v'_2(x_0)$. Therefore by the uniqueness result for O.D.E. we get $v_1 \equiv v_2$.  This proves the first statement of Proposition \ref{p4per}. 

It remains to show that a global minimizer $v$ either does not intersect an odd minimizer $u$ or coincides with it. Let $\hat v(x)=-v(-x)$ be the symmetric global minimizer. We first notice that if a global minimizer $v$ is odd and $\alpha=0$, it holds $v\equiv 0$ (cf. Proposition \ref{p1per} (i)), and necessarily the odd minimizer $u$ is also identically $0$ (cf. Proposition \ref{p2per}), thus $v\equiv u$. Otherwise, if $\alpha>0$ and $v$ is odd, we still have  $v\equiv u$, by uniqueness of $u$.  As previously, if $v$  intersects $u$ twice on $[0,T/2)$ or on $[T/2,T)$, we obtain $v\equiv u$. Similarly, if $v$ intersects $u$ only once on $[0,T)$, we reach the same conclusion. 
Now, suppose that $v$ and $u$ intersect 
\begin{itemize}
\item only once on $[T/2,T)$, that is, only once on $[-T/2,0)$ at a point $x_1 \in [-T/2,0)$, 
\item and only once on $[0,T/2)$ at a point $x_2 \in [0,T/2)$. 
\end{itemize}
The case where $v-u$ does not change sign at $x_i$ for some $i=1,2,$ implies that $v\equiv u$, since it holds $v(x_i)=u(x_i)$ as well as $v'(x_i)=u'(x_i)$. Thus, it remains to examine the case where $v-u$ changes sign at $x_1$ and $x_2$. Without loss of generality we assume that $v(x)<u(x) \Leftrightarrow x \in [-T/2,T/2)\setminus [x_1,x_2]$, and distinguish three subcases. 1) If $x_1=-x_2$ (or $x_1=-T/2$), it follows that $v\equiv \hat v$, since $v(x_1)=\hat v(x_1)$ (resp. $v(-T/2)=0=\hat v(-T/2)$), and $v$ is odd. Thus, $v\equiv u$ in view of our remark. 2) If $x_2<|x_1|<T/2$, we have $v(-x_2)>\hat v(-x_2)=u(-x_2)$, and $v(-x_1)<\hat v(-x_1)=u(-x_1)$. This implies that $v$ intersects $\hat v$, and thus $v\equiv \tilde v\equiv u$. 3) Assuming that $x_2>|x_1|$, we have $v(-x_2)<\hat v(-x_2)=u(-x_2)$, and $v(-x_1)>\hat v(-x_1)=u(-x_1)$. This implies again that $v$ intersects $\hat v$, and thus $v\equiv \tilde v\equiv u$. Therefore, the proof of the theorem is complete.
\end{proof}

\begin{proposition}\label{p3per}(Zeros of the global minimizer)
\begin{itemize}
\item[(i)] Assuming that $\alpha>0$, the minimizers $v$ constructed in Theorem \ref{th1} 
 have at most two zeros within a period.
\item[(ii)] If assumption (B) holds, and if $\epsilon$ and $\alpha>0$ remain in a bounded interval, there exists a constant $\delta>0$ such that when $\frac{\epsilon}{\alpha}<\delta$, the minimizers $v$ have within a period exactly two zeros, $\bar x_1$ and $\bar x_2$. Moreover, $v(x)>0  \Leftrightarrow x \in (\bar x_1,\bar x_2)+T\Z$, with
$|\bar x_1| \leq \xi+\mathcal{O }\big(\sqrt{\epsilon/\alpha}\big)$, and $|\bar x_2-T/2|\leq \mathcal{O }\big((\epsilon/\alpha)^\frac{\beta}{\beta+2}\big)$.
\end{itemize}
\end{proposition}
Next, we renormalize the energy by setting $\mathcal{E}(u):=E(u)+\int_{0}^{T}\frac{(\mu^+)^2}{4\epsilon}$. Under assumption (A), we have the upper bounds:  
$$\limsup_{\epsilon\to 0, \alpha\to \alpha_0}\mathcal{E}(v_{\epsilon,\alpha})\leq \frac{2}{3}\min\big(0, (\sqrt{2}-\alpha_0)(\mu(0))^{3/2}+(\sqrt{2}+\alpha_0)(\mu(T/2))^{3/2}\big),$$
while under assumption (B), we obtain:
$$\limsup_{\epsilon\to 0, \alpha\to \alpha_0}\mathcal{E}(v_{\epsilon,\alpha})\leq  \frac{2(\mu(0))^{3/2}}{3}\min\big(0, \sqrt{2}-\alpha_0\big).$$
From these bounds, proceeding as in Theorems \ref{p10sh} and \ref{p10shter}, we deduce:

\begin{theorem}\label{p10shquaper} 
Assume (A), let $\alpha^{**}=\sqrt{2}\frac{(\mu(0))^{3/2}+(\mu(T/2))^{3/2}}{(\mu(0))^{3/2}-(\mu(T/2))^{3/2}}$, and let $v_{\epsilon,\alpha}$ be global minimizers. Then, 
\begin{itemize}
\item[(i)] as $\epsilon \to 0$ and $\alpha \to \alpha_0 \in (0,\alpha^{**})$, $v_{\epsilon,\alpha}$ does not vanish,
\item[(ii)] as $\epsilon \to 0$ and $\alpha \to \alpha_0 \in (\alpha^{**},\infty)$, $v_{\epsilon,\alpha}$ has up to the symmetry $\hat v=-v(-x)$, two zeros $\bar x_1(\epsilon,\alpha)<\bar x_2(\epsilon,\alpha)$ such that $0\geq\bar x_1(\epsilon,\alpha) \to 0$, and $T/2\leq\bar x_2(\epsilon,\alpha)\to T/2$. Moreover, $v(x)>0  \Leftrightarrow x \in (\bar x_1,\bar x_2)+T\Z$.
\end{itemize}

\end{theorem}
\begin{theorem}\label{p10shterper} 
Assume (B), and let $v_{\epsilon,\alpha}$ be global minimizers. Then,
as $\epsilon \to 0$ and $\alpha \to \alpha_0>0$, $v_{\epsilon,\alpha}$ has up to the symmetry $\hat v(x)=-v(-x)$, two zeros $\bar x_1(\epsilon,\alpha)<\bar x_2(\epsilon,\alpha)$ (cf. Proposition \ref{p3per} (ii)) such that 
\begin{itemize}
\item[(i)]when $\alpha \to \alpha_0 \in (0,\sqrt{2})$, $\bar x_1(\epsilon,\alpha) \to -\xi $, and $T/2\leq\bar x_2(\epsilon,\alpha)\to T/2$,
\item[(ii)] when $\alpha \to \alpha_0 \in (\sqrt{2},\infty)$, $0 \geq\bar x_1(\epsilon,\alpha) \to 0$, $T/2\leq \bar x_2(\epsilon,\alpha)\to T/2$. 
\end{itemize}
\end{theorem}
\begin{remark}
Unlike the case where $\mu$ has two bumps (cf. Theorem \ref{p10shter}), when $\mu$ is periodic the location of the zeros of the minimizers is precisely determined in Theorem \ref{p10shterper}. Also note that when $\mu>0$ is periodic, there is a unique threshold $\alpha^{**}$ ruling the regime of the minimizers (cf. Theorem \ref{p10shquaper}), which has a similar expression as the second threshold value of $\alpha$ appearing in Theorem \ref{p10sh}.

\end{remark}
\begin{remark}
Theorem \ref{p4c} does not hold in general when $\mu$ is periodic. Indeed, if it holds $\int_0^T \mu>0$, one can see that when $\alpha=0$ and $\epsilon >0$, we have for a small constant $k$:
$$E(v)\leq E(k)=\frac{k^2}{4\epsilon}\Big( -2\int_0^T \mu(x)\dd x +k^2 \Big)<0,$$
thus the global minimizer is never identically zero.
On the other hand, when $\alpha=0$, the odd mnimizer $u$ is identically $0$ if and only if $\epsilon\geq\epsilon_1$, where 
$$
0<\epsilon_1: =\sup_{ \phi \in H^1_0([0,T/2], \phi \neq 0}\frac{(\int_0^{T/2}\mu(x)\phi^2(x)\dd x)^{1/2}}
{(\int_{\R}|\phi_x|^2)^{1/2}}<\infty.$$
\end{remark}
\begin{remark}
However, Theorem \ref{p4d} still holds when $\mu$ is periodic, and we still have the convergence in $C^1_{\mathrm{ loc}}(\R)$ of the rescaled odd minimizers:
\begin{equation}
\lim_{ \alpha\to \infty} \alpha^{-1/3}u_{\epsilon,\alpha}(x_0+\alpha^{-1/3}s)=(\epsilon f(x_0))^{1/3}, \ \forall x_0\in \R \text{ fixed}.\nonumber
\end{equation} .
\end{remark}

\section{Appendix}

In this appendix, we state two convergence results for the solutions of \eqref{digital}, under assumptions \eqref{hyp3}. The first one (cf. Proposition \ref{pgia}) which applies in the region where $\mu<0$, shows that as $\epsilon\to 0$, O.D.E. \eqref{digital} can be approximate by the linear equation $\mu(x) u+\ve \alpha f(x)=0$. The second one (cf. Proposition \ref{psta}) establishes that as $\alpha\to\infty$, with $\epsilon>0$ fixed, O.D.E. \eqref{digital} can be approximate by the cubic equation $-u^3+\ve \alpha f(x)=0$. Next, we state two Poincar\'{e} type inequalities (cf. Proposition  \ref{p4b}), from which we will deduce in Theorems \ref{p4c} and \ref{p4d}, uniqueness results for solutions of \eqref{digital}. Except Proposition \ref{p4b} (i), similar versions of this material can be found in our previous work.

\begin{proposition}\label{pgia}(cf. Theorem 1.1 (iii) of \cite{symb} and \cite{erg13} for a similar result)
For $\epsilon$ and $\alpha$ belonging to a bounded interval, let $u_{\epsilon,\alpha}$ be a solution of \eqref{digital} converging to $0$ at $\pm\infty$. Then, given $x_0\in \{x\in\R:\mu(x)<0\}$ and a compact interval $I\subset\{x\in\R:\mu(x)<0\}$, it holds
\begin{itemize}
\item[i)] $\lim_{\epsilon\to 0, \alpha\to \alpha_0} \frac{u_{\epsilon,\alpha}(x_0+s\epsilon)}{\epsilon}=-\frac{\alpha_0 f(x_0)}{\mu(x_0)}$ in $C^1_{\mathrm{ loc}}(\R)$, 
\item[ii)] $\lim_{\epsilon\to 0, \alpha\to \alpha_0} \frac{u_{\epsilon,\alpha}(x)}{\epsilon}=-\frac{\alpha_0 f(x)}{\mu(x)}$, uniformly on $I$.
\end{itemize}
\end{proposition}
\begin{proof}
To prove the proposition, it suffices to consider sequences $\{\epsilon_n\}$, $\{\alpha_n\}$, and $\{x_n\}\subset I$, converging respectively to $0$, $\alpha_0$, and $x_0$, and show that the rescaled functions $\tilde u_n(s):=\frac{u_{\epsilon_n,\alpha_n}(x_n+s\epsilon)}{\epsilon}$ converge in $C^1_{\mathrm{ loc}}(\R)$ up to subsequence, to the constant $-\frac{\alpha_0 f(x_0)}{\mu(x_0)}$. Indeed, in view of Proposition \ref{p0} (ii), the functions $\tilde u_n(s)$ are uniformly bounded on compact intervals. Moreover, they satisfy the O.D.E.: 
\begin{equation}\label{odegiant2}
\tilde u_n''(s)+\mu(x_n+\epsilon_n s)\tilde u_n(s)-\epsilon_n^2\tilde u_n^3(s)+\alpha_n f(x_n+\epsilon_n s)=0.
\end{equation}
It follows that $\tilde u_n''$ as well as $\tilde u_n'$ (cf. \S 3.4 p. 37 of \cite{gilb}) are also uniformly bounded on compact intervals. Thus, we can apply the theorem of Ascoli to deduce that up to subsequence $\tilde u_n(s)$ converges in $C^1_{\mathrm{ loc}}(\R)$ to a function $\tilde U(s)$ solving the linear O.D.E.:  
\begin{equation}\label{odegiant3}
\tilde U''(s)+\mu(x_0)\tilde U(s)+\alpha_0 f(x_0)=0.
\end{equation}
In addition, Proposition \ref{p0} (iii) implies that $\tilde U$ is bounded. This proves that $\tilde U$ is the constant function $-\frac{\alpha_0 f(x_0)}{\mu(x_0)}$.
\end{proof}

\begin{proposition}\label{psta}
For fixed $\epsilon>0$ and $\alpha_0>0$, and $\alpha\in(\alpha_0,\infty)$, let $u_{\epsilon,\alpha}$ be a solution of \eqref{digital} converging to $0$ at $\pm\infty$. Then, given $x_0\in \R$ and a compact interval $I\subset\R$, it holds
 \begin{itemize}
\item[i)] $\lim_{ \alpha\to \infty} \alpha^{-1/3}u_{\epsilon,\alpha}(x_0+\alpha^{-1/3}s)=(\epsilon f(x_0))^{1/3}$ in $C^1_{\mathrm{ loc}}(\R)$.
\item[ii)] $\lim_{ \alpha\to \infty} \alpha^{-1/3}u_{\epsilon,\alpha}(x)=(\epsilon f(x))^{1/3}$, uniformly on $I$.
\end{itemize}

\end{proposition}
\begin{proof}

We refer to the proof of  \cite[Lemma 3.10]{symb}, which is similar.
\end{proof}

\begin{proposition}\label{p4b}
(i) Let $\mu\in L^{\infty}(\R)$ be such that the set $I_\eta:=\{x\in\R: \ \mu(x)+\eta>0\}$ is bounded for some constant $\eta>0$. Then, there exists a constant $C>0$ such that
\begin{equation}
\label{Poincare}
\int_{\R}\mu(x)\phi^2(x)\dd x \leq C \int_{\R}|\phi'|^2,\ \forall \phi \in H^1(\R).
\end{equation}
(ii) Furthermore, if for every $\alpha>0$, $\mu_\alpha:\R\to\R$ is a measurable function satisfying $\mu_\alpha\leq \mu$, and $\mu_\alpha \to -\infty$ a.e., as $\alpha\to \infty$, then given $\epsilon>0$, there exists $A>0$ such that for every $\alpha >A$ we have
\begin{equation}
\label{Poincarebis}
\int_{\R}\mu_\alpha(x)\phi^2(x)\dd x < \epsilon^2 \int_{\R}|\phi'|^2,\ \forall \phi \in H^1(\R), \, \phi \neq 0.
\end{equation}
\end{proposition}
\begin{proof}
By homogeneity, it is sufficient to prove \eqref{Poincare} for $\left\|\phi \right\|_{H^{1}(\R)}=1$. Suppose by contradiction that for every $n\geq 1$, there exists $\phi_n \in H^{1}(\R)$ such that 
\begin{equation}\label{Poincare2}
\frac{1}{n}\int_{\R}\mu(x)\phi^2_n(x)\dd x \geq \int_{\R}|\phi'_n|^2. 
\end{equation}
This implies that $\lim_{n \to\infty}\int_{\R}|\phi'_n|^2=0$ and $\lim_{n \to\infty}\int_{\R}|\phi_n|^2=1$. Since
$\left\|\phi_n \right\|_{H^{1}(\R)}$ is bounded, we can extract a subsequence, still called $\phi_n$, such that $\phi_n\rightharpoonup \psi$ weakly in $H^1(\R)$, and $\phi_n\to \psi$ in $L^2_{\mathrm{loc}}(\R)$. By lower semicontinuity we have $\int_{\R} |\psi'|^2\leq \liminf \int_{\R}  |\phi'_n|^2=0$, thus $\psi\equiv0$. Moreover, since $\lim_{n \to\infty}\int_{I_\eta}|\phi_n|^2=0$, we deduce that $\lim_{n \to\infty}\int_{\R\setminus I_\eta}|\phi_n|^2=1$.
Consequently, it holds $\int_{\R}\mu|\phi_n|^2\leq \left\|\mu \right\|_{L^{\infty}(\R)}\int_{I_\eta}|\phi_n|^2-\eta\int_{\R\setminus I_\eta}|\phi_n|^2<0$, for $n$ large enough, which contradicts \eqref{Poincare2}. 

For the proof of \eqref{Poincarebis}, we refer to \cite[Lemma 3.9]{symb}.
\end{proof}

\section*{Acknowledgements}
M.G.C and M.D.Z. thank for the financial support of
ANID-Millenium Science Initiative Program-ICN17 012,
Chile (MIRO). M.K. thanks for the financial support
of FONDECYT 1250156 and CMM ANID project
FB210005. The research project of P.S. is implemented within the framework of H.F.R.I call
``Basic research Financing (Horizontal support of all Sciences)'' under the National Recovery and Resilience Plan ``Greece 2.0'' funded by the European Union - NextGenerationEU (H.F.R.I. Project Number: 016097).

\end{document}